\documentclass[a4paper,oneside,reqno]{amsart}

\usepackage[english]{babel}
\usepackage[utf8]{inputenc}		
\usepackage[scaled]{helvet}		
\usepackage{courier}			
\usepackage{eulervm}			
\usepackage[bb=boondox,bbscaled=1.05,scr=dutchcal]{mathalfa}	%
\usepackage{dsfont}
\normalfont

\usepackage{indentfirst}				
\usepackage{tabularx,booktabs}			
\usepackage{caption,subcaption}			
\usepackage{fullpage}				    
\usepackage[english]{varioref}			
\usepackage[dvipsnames]{xcolor}			
\usepackage{hyperref}					
\usepackage{bookmark}					
\usepackage{textcomp}
\usepackage{rotating,afterpage,pdflscape}
\usepackage{enumerate}

\definecolor{webbrown}{rgb}{0.65, 0.16, 0.16}

\hypersetup{
colorlinks=true, linktocpage=true, pdfstartpage=1, pdfstartview=FitV,breaklinks=true, pdfpagemode=UseNone, pageanchor=true, pdfpagemode=UseOutlines,plainpages=false, bookmarksnumbered, bookmarksopen=true, bookmarksopenlevel=1,hypertexnames=true, pdfhighlight=/O,urlcolor=webbrown, linkcolor=RoyalBlue, citecolor=ForestGreen}

\usepackage{graphicx}		
\usepackage{float}
\usepackage{tikz}			
\usepackage{tikz-cd}		
\usetikzlibrary{
arrows,decorations.pathreplacing,decorations.markings,shapes.geometric,positioning,calc,fadings,angles
	}
\tikzfading[name=inner fade, inner color=transparent!0, outer color=transparent!100]
\graphicspath{{Images/}}

\usepackage[textsize=footnotesize,textwidth=0.85in]{todonotes}
\presetkeys%
    {todonotes}%
    {color=Dandelion}{}%
\usepackage{amsmath,amssymb,amsthm,mathtools}
\usepackage{bm,braket,faktor,stmaryrd}

\numberwithin{equation}{section}

\newcommand{\ZZ}{\mathbb{Z}}

\newcommand{\ch}{{\rm ch}}

\newcommand{\dd}{\mathrm{d}}

\newcommand{\MM}{\overline{\mathcal{M}}}

\newcommand{\Res}{\mathop{\mathrm{Res}}}

\theoremstyle{plain}
	\newtheorem{theorem}{Theorem}
	\newtheorem*{theorem*}{Theorem}
	\newtheorem{proposition}[theorem]{Proposition}
	\newtheorem{corollary}[theorem]{Corollary}
	\newtheorem{lemma}[theorem]{Lemma}
	
	\numberwithin{theorem}{section}
	\newtheorem{theoremA}{Theorem}

\newtheorem{theoremB}{Theorem}

\newtheorem{theoremC}{Theorem}

\theoremstyle{definition}
	\newtheorem{definition}[theorem]{Definition}

\theoremstyle{remark}
	\newtheorem{remark }{Remark}

\makeatletter
\renewcommand\section{\@startsection {section}{1}{\z@}%
                                   {-3.5ex \@plus -1ex \@minus -.2ex}%
                                   {2.3ex \@plus.2ex}%
                                   {\centering\Large\scshape}}
\makeatother

\title{
On ELSV-type formulae and relations between $\Omega$-integrals via deformations of spectral curves
}

\author{Ga\"{e}tan Borot}
\address{G.~B.: 
	Humboldt-Universit\"at zu Berlin, Institut f\"ur Mathematik und Institut f\"ur Physik, Unter den Linden 6, 10099 Berlin, Germany.}
\email{gaetan.borot@hu-berlin.de}

\author{Maksim Karev}
\address{M.~K.: 
	Guangdong-Technion Israel Institute of Technology, 241 Daxue Road, Guangdong, Shantou, 515063, China
}
\email{maksim.karev@gtiit.edu.cn}

\author{Danilo Lewa\'{n}ski }
\address{D.~L.: 
	Universit\'a di Trieste, Dipartimento di Matematica, Informatica e Geoscienze, via Weiss 2, 34128 Trieste, Italy
}
\email{danilo.lewanski@units.it}

\date{}


\begin{document}

\begin{abstract}
The general relation between Chekhov--Eynard--Orantin topological recursion and the intersection theory on the moduli space of curves, the deformation techniques in topological recursion, and the polynomiality dependency of its deformation parameters can be combined to derive vanishing relations involving intersection indices of tautological classes. We apply this strategy to three different families of spectral curves and show they give vanishing relations for integrals involving $\Omega$-classes. The first class of vanishing relations are genus-independent and generalises the vanishings found by Johnson--Pandharipande--Tseng \cite{JPT} and by the authors jointly with Do and Moskovsky \cite{BDKLM}. The two other class of vanishing relations are of a different nature and depend on the genus.
\end{abstract}

\maketitle
\tableofcontents

\section{Introduction}
\label{sec:Intro}

\vspace{0.5cm}

\subsection{Background: intersection theory and topological recursion}

The $\Omega$-classes $\Omega^{(r,s)}_{g;a_1,\ldots,a_n}$ are generalisations of the Hodge class on the moduli space of complex curves $\overline{\mathcal{M}}_{g,n}$, depending on integer parameters $(r,s) \in \mathbb{Z}_{> 0} \times \mathbb{Z}$ and $a_1,\ldots,a_n \in \mathbb{Z}$. They were introduced by Mumford for $s=1$ and $r=1$ \cite{Mum83}, by Bini \cite{Bin03} for generic $s$ and $r = 1$, by Chiodo \cite{Chi08} in general, and enjoy many interesting properties. In particular, they form a cohomological field theory, \textit{i.e.} they are compatible with the boundary structure of the moduli spaces \cite{lew-pop-sha-zvo17}. They appear in the enumerative geometry of branched coverings over $\mathbb{P}^1$, aka Hurwitz theory, as well as in Masur--Veech volumes strata of the moduli of quadratic differentials \cite{CMS+19}. More precisely, the pioneering work of Ekedahl--Lando--Shapiro--Vainshtein \cite{ELSV} gave a formula for simple Hurwitz numbers via intersection indices of the Hodge class, and this was generalised to a subclass of double Hurwitz numbers by Johnson--Pandharipande--Tseng \cite{JPT} involving virtual localisation on the moduli space of stable maps to an orbifold $\mathbb{P}^1$. Later, ELSV-type formulae have been found for many other types of weighted Hurwitz numbers. In particular, the Hurwitz number of type ``$s$-orbifold with $(d + 1)$-completed cycles'' can be expressed via intersection indices of $\Omega^{(ds,s)}$ \cite{kra-lew-pop-sha19,bor-kra-lew-pop-sha20,dun-kra-pop-sha19-2}.

\medskip

Topological recursion \cite{eyn-ora-07} is a procedure that takes a spectral curve $\mathscr{S}$ (\textit{i.e.} a fixed branched cover of complex curves) as input and gives a collection $\omega_{g,n}$ of multidifferentials on $\mathscr{S}^n$ indexed by $(g,n) \in \mathbb{Z}_{\geq 0} \times \mathbb{Z}_{> 0}$, called correlators, as output. For appropriate choices of spectral curves the resulting correlators store enumerative information, and since its formalisation by Chekhov, Eynard and Orantin, the structure of topological recursion has appeared ubiquitous in enumerative geometry \cite{EynardICM}. In the present context there are two results of central importance. First, after the efforts of many, leading to the treatment of several special cases \cite{eyn-mul-saf11,do-lei-nor16,bou-her-liu-mul14,do-man14,do-dye-mat17,do-kar16,Lew18,dun-kra-pop-sha19,do-kar18,BDKLM,bouchard2023topological}, weighted Hurwitz numbers have been shown to be quite generally governed by topological recursion \cite{ACEH3,bych-dun-kaz-sha}. Second, the coefficients of the $\omega_{g,n}$ can always be expressed as intersection theory of a tautological class on $\overline{\mathcal{M}}_{g,n}$, which can be constructed from the spectral curve and which we call the TR class \cite{eyn14}, see \textit{e.g.} \cite[Section 7]{BKS} for a review. The ELSV and JPT formulae originally obtained via algebraic geometry have a second proof by combining these two results \cite{dun-lew-pop-sha15}, but for other ELSV-like formulae the path through topological recursion remains the only available proof to this day, \textit{e.g.}
\begin{itemize} 
\item  Topological recursion for weakly monotone Hurwitz numbers \cite{do-dye-mat17,do-kar16} implies a formula for these numbers in terms of intersection indices of $\kappa$-classes \cite{do-kar16,ale-lew-sha16};
\item  Zvonkine's conjecture \cite{Zvo06,SSZ15} proved in \cite{bor-kra-lew-pop-sha20,dun-kra-pop-sha19-2}:  ``$s$-orbifold, $(d + 1)$-completed cycles'' in terms of intersection indices of the class $\Omega^{(ds,s)}$.
\item The generalisation of the JPT formula proved in \cite{BDKLM}: all double Hurwitz numbers in terms of the classes $\Omega^{(r,r)}$.
\end{itemize}

An interesting feature of topological recursion is the possibility it offers to study deformations, which are controlled by global properties of families of spectral curves. This was applied in \cite{BDKLM} to derive an ELSV-like formula for double Hurwitz numbers by realising them as deformations of $s$-orbifold Hurwitz numbers and following this deformation through the topological recursion procedure. Double Hurwitz numbers are not just numbers, but rather polynomials in a finite set of variables, such that each monomial keeps track of covers with a different ramification profile above $\infty$. The deformation in the context of topological recursion produces double Hurwitz numbers as Laurent polynomials with coefficients being intersection indices of the class $\Omega^{(s,s)}$. This means that all the coefficients in front of negative powers should vanish. As a result, this has produced vanishing relations for intersection indices, which generalise the ones found by Johnson--Pandharipande--Tseng \cite{JPT}.

\vspace{0.2cm}

\subsection{Main results}

In this article we explore further the potential of deformation techniques in the framework of topological recursion, in order to derive new relations between integrals of $\Omega$-classes agains $\psi$-classes, in the following referred to as $\Omega$-integrals. This initiates a more structured analysis of the vanishing of $\Omega$-integrals, although much remains to be understood.

\medskip

After a background section, Section~\ref{sec:Omega:classes}, introducing the $\Omega$-class and relevant facts from Hurwitz theory and topological recursion, we exhibit in Sections~\ref{sec:Deformation:one}-\ref{sec:Deformation:two}-\ref{sec:Deformation:three} three families of spectral curves with the following properties:
\begin{itemize}
\item[(i)] the TR class associated to the central fiber is (up to a rescaling) an $\Omega$-class;
\item[(ii)] the general deformation properties of topological recursion --- summarised in Section~\ref{sec:principle} --- compute the correlators for a generic fiber as a Laurent series (in the parameters of the family) of $\Omega$-integrals;
\item[(iii)] we know from other means that this series can only contain nonnegative powers of the parameters.
\end{itemize}

We obtain relations between $\Omega$-integrals by equating to $0$ the coefficients in (ii) of monomials involving negative powers due to the crucial and non-trivial property (iii). The two first families of spectral curves we consider govern weighted Hurwitz numbers and (iii) is then a consequence of the structural results of \cite{bych-dun-kaz-sha}. The third family falls out of the scope of \cite{bych-dun-kaz-sha} but we justify in Proposition~\ref{polynom3} that it satisfies (iii) by coming back to a careful analysis of the topological recursion formula. We summarise our findings below, where we introduced the rising factorial
\begin{equation}
\label{risingfac}
[x]_{\ell} = x(x +1)\cdots (x + \ell - 1)
\end{equation}
and all the other relevant notations to read these formulae are explained in the next Section~\ref{Sec:not}.

\medskip

Our first set of relations holds in the divisible case $r = ds$ and is a direct generalisation of the one found for $d = 1$ \cite{BDKLM} and in particular generalises the set of vanishings found in \cite{JPT}, see the discussion in Section~\ref{sec:Omega:vanishing}.

\begin{theoremA}[Section~\ref{sec:Deformation:one}]
\label{th:A}
Let $d \geq 1$, $s \geq 2$ and set $r = ds$. Let $g \geq 0$. For any non-empty partition $\mu$ and non-empty $(s - 1)$-partition $\nu$ such that $|\mu| + |\nu| \leq s(\ell(\mu) - 1)$, we have
\begin{equation*}
\sum_{k = 1}^{\ell(\nu)} \frac{(-1)^{\ell(\nu) - k} (r/s)^{k}}{k!} 
\sum_{\substack{\boldsymbol{\rho} \in \mathcal{P}_{s - 1}^{k} \\ \sqcup \boldsymbol{\rho} = \nu}} \prod_{c = 1}^{k} 
\frac{\big[\frac{s - |\rho^{(c)}|}{s}\big]_{\ell(\rho^{(c)}) - 1}}{|\text{Aut}(\rho^{(c)})|} \int_{\overline{\mathcal{M}}_{g,\ell(\mu) + k}} \frac{\Omega^{(r,s)}_{g;-\overline{\mu},s - |\boldsymbol{\rho}|}}{\prod_{i = 1}^{\ell(\mu)} \left(1 - \frac{\mu_i}{r}\psi_i\right)} = 0.
\end{equation*}
If furthermore $\nu$ is bounded, \textit{i.e.} $\min_{i \neq j} (\nu_i + \nu_j) \geq s$, this sum has a single term and we get the vanishing
\begin{equation*}
\int_{\overline{\mathcal{M}}_{g,\ell(\mu) + \ell(\nu)}} \frac{\Omega^{(r,s)}_{g;-\overline{\mu},s - \nu}}{\prod_{i = 1}^n \left(1 - \frac{\mu_i}{r}\psi_i\right)} = 0.
\end{equation*}
\end{theoremA}

The second relation we find again concerns the divisible case $r = ds$ and has a range of vanishing depending on the genus. It is presented in the text in Theorem~\ref{genev}, but as it is rather complicated, we only state here the $s = 1$ case.

\begin{theoremB}[Section~\ref{sec:Deformation:two}] \label{th:B} Let $r \geq 2$, and $g \geq 0$. For any non-empty partition $\mu$ and possibly empty $(r - 1)$-partition $\tau$ such that
$$
2g - 2 + \ell(\mu) + |\mu| + |\tau|  \leq d(\ell(\tau) - 1),
$$
we have
\begin{equation*}
\sum_{h = 1}^{\ell(\tau)} \frac{(-1)^{\ell(\tau) - h}}{h!} \sum_{\substack{\boldsymbol{\rho} \in \mathcal{P}_{r - 1}^{h} \\ 2 \leq |\rho^{(c)}| \leq r - 1 \\ \sqcup \boldsymbol{\rho} = \tau}} \prod_{c = 1}^{h} \frac{\big[\frac{r + 1 - |\rho^{(c)}|}{r}\big]_{\ell(\rho^{(c)}) - 1}}{|\text{Aut}(\rho^{(c)})|} \int_{\overline{\mathcal{M}}_{g,\ell(\mu) + h}} \frac{\Omega^{(r,1)}_{g;-\overline{\mu},r + 1 - |\boldsymbol{\rho}|}}{\prod_{i = 1}^{\ell(\mu)} \left(1 - \frac{\mu_i}{r}\psi_i\right)} = 0.
\end{equation*}
If furthermore $\tau$ has no parts $1$ and is bounded, \textit{i.e.} $\min_{i \neq j} (\tau_i + \tau_j) \geq r$, this sum has a single term and we get the vanishing
$$
\int_{\overline{\mathcal{M}}_{g,\ell(\mu) + \ell(\tau)}} \frac{\Omega^{(r,1)}_{g;-\overline{\mu},r - \tau}}{\prod_{i = 1}^{\ell(\mu)} \left(1 - \frac{\mu_i}{r} \psi_i\right)}.
$$  
\end{theoremB}

The third and last relation we find holds for positive $r$ and $s$ without divisibility condition, with a range of vanishing depending on the genus.

\begin{theoremC}[Section~\ref{sec:Deformation:three}]
\label{th:C}
Let $r \geq 2$ and $s \in \{1,\ldots,r - 1\}$. Let $g \geq 0$. For any non-empty partition $\mu$ and a non-empty $(r + s - 1)$-partition $\tau$ whose parts belong to $\{s,\ldots,r + s - 1\}$ and such that
$$
(2g -2  +\ell(\mu))s + |\mu| + |\tau| \leq (r + s)\ell(\tau) - r,
$$
we have
$$
\sum_{k = 1}^{\ell(\tau)} \frac{1}{k!} \sum_{\substack{\boldsymbol{\rho} \in \mathcal{P}_{s,r+s-1}^{k} \\ |\rho^{(c)}| \leq r + s - 1 \\ \sqcup \boldsymbol{\rho} = \nu}} \prod_{c = 1}^{k} \frac{\big[\frac{r + s - |\rho^{(c)}|}{r}\big]_{\ell(\rho^{(c)}) - 1}}{|\text{Aut}(\rho^{(c)})|} \int_{\overline{\mathcal{M}}_{g,\ell(\mu) + k}} \frac{\Omega_{g;-\overline{\mu},r + s - |\boldsymbol{\rho}|}}{\prod_{i = 1}^{\ell(\mu)} \left(1 - \frac{\mu_i}{r}\psi_i\right)} = 0.
$$
If furthermore $\tau$ is bounded, \textit{i.e.} $\min_{i \neq j} (\tau_i + \tau_j) \geq r + s$, this sum has a single term and we get the vanishing
$$
\int_{\overline{\mathcal{M}}_{g,\ell(\mu) + \ell(\tau)}} \frac{\Omega^{(r,s)}_{g;-\overline{\mu},r + s - \tau}}{\prod_{i = 1}^{\ell(\mu)} \left(1 - \frac{\mu_i}{r} \psi_i\right)} = 0.
$$  
\end{theoremC}

In Section~\ref{sec:Omega:vanishing}, we give two other previously known vanishing results for $\Omega$-integrals, as well as observations from our numerical experiments carried out with the package \textsc{Admcycles} \cite{admcycles}. In Appendix~\ref{sec:Omega:properties} we collect basic properties of the $\Omega$-classes.

\vspace{0.2cm}

\subsection{Comments}

This note can be thought as a guide through the application of deformation techniques in topological recursion to obtain consequences in enumerative geometry, especially in the richer case of oblique deformations (see the terminology in Section~\ref{sec:principle}). Two recent works where deformations in topological recursion are studied should be mentioned.

\medskip

 In \cite{BCCGF22} topological recursion is established for a large class of maps and constellation enumerations, by first showing it for a restricted model, and then using deformations and checking that topological recursion gets carried along. A notable difference with our work is that the deformations they use are horizontal (see Definition~\ref{defclass}), thus the combinatorics of the Taylor expansion are simpler than the general ones discussed in Section~\ref{sec:principle}.

\medskip

 In \cite{TRLimits} the question of analyticity of topological recursion correlators along families of spectral curves (allowing certain catastrophes) is addressed, and sufficient conditions are proposed for analyticity to hold. Yet, in the families considered in the present work, ramification points can approach the essential singularity of the logarithm, and this case is not covered by \cite{TRLimits}.  We therefore have to rely on different arguments to justify the regularity needed to get our vanishing results. The need to consider logarithms come from the particular setting of the $\Omega$-classes. However, following the same strategy, it should be possible to derive vanishing results for intersection indices of  other tautological classes as consequence of the analyticity results proved in \cite{TRLimits}

\vspace{0.2cm}

\subsection{Notation}
\label{Sec:not}

If $m \in \mathbb{Z}$, we denote $-\overline{m}$ the unique integer in $\{0,\ldots,d - 1\}$ such that $-m = - \overline{m} \,\,\text{mod}\,\,r$. By extension, if $\mu$ is a partition, $-\overline{\mu}$ is the tuple $(-\overline{\mu}_1,\ldots,-\overline{\mu}_{\ell(\mu)})$. 

\medskip

A partition $\lambda$ is a (possibly empty) finite sequence of positive integers $\lambda_1 \geq \cdots \geq \lambda_{\ell}$, called parts. Its length is $\ell(\lambda) = \ell$, its size is $|\lambda| = \sum_{i = 1}^{\ell} \lambda_i$. The automorphism group $\text{Aut}(\lambda)$ is the set of permutations of $\{1,\ldots,\ell(\lambda)\}$ respecting the weak decreasing order in $\lambda$. The notation $\lambda = (1^{m_1} 2^{m_2} \cdots)$ means that $\lambda$ is the partition with $m_1$ parts $1$, $m_2$ parts $2$, etc. We have $|\text{Aut}(\lambda)| = \prod_{i \geq 1} m_i!$.  
The notation $\lambda \vdash N$ means that $\lambda$ is a partition of size $N$.   

\medskip

We say that $\lambda$ is an $N$-partition if it is empty or $\lambda_1 \leq N$. We denote $\mathcal{P}_{N}$ the set of $N$-partitions, and $\mathcal{P}\!\!\mathcal{P}_{N} \subseteq \mathcal{P}_{N}$ the set of $N$-partitions of size $\leq N$. If $\lambda$ is an $(N - 1)$-partition, we write $N -  \lambda$ the partition $N - \lambda_{\ell(\lambda)},\ldots,N - \lambda_1$. We say that an $(N - 1)$-partition is \emph{bounded} if $\min_{i \neq j} (\lambda_i + \lambda_j) \geq N$. A $(N',N)$-partition is a partition whose parts all belong to $\{N',N' + 1,\ldots,N\}$. We denote $\mathcal{P}_{N',N}$ the set of $(N',N)$-partitions. A more complicated set of ordered pairs of partitions is introduced in Definition~\ref{thetrho} for the needs of Theorem~\ref{genev}.

\medskip

An extended $N$-partition is $\overline{\lambda} = (1^{m_1} \cdots N^{m_N})$ with $m_1,\ldots,m_{N - 1} \in \mathbb{Z}_{\geq 0}$ but $m_N \in \mathbb{Z}$. Its size is defined as usual $|\overline{\lambda}| = \sum_{i = 1}^{N} im_i$. The $(N - 1)$-partition associated to $\overline{\lambda}$ is  $\lambda = (1^{m_1} \cdots (N - 1)^{m_{N - 1}})$. We denote $\overline{\mathcal{P}}_{N}$ the set of extended $N$-partitions. 

\medskip

If $p_1,\ldots,p_N$ is a $N$-tuple of variables and $\overline{\lambda} = (1^{m_1} \cdots N^{m_N})$ is an extended $N$-partition, we denote
$$
\vec{p}_{\overline{\lambda}} = p_{1}^{m_1} \cdots p_N^{m_N}.
$$

\medskip

\vspace{0.2cm}

\subsection*{Acknowledgements}

This work is partly a result of the ERC-SyG project, Recursive and Exact New Quantum Theory (ReNewQuantum) which received funding from the European Research Council (ERC) under the European Union's Horizon 2020 research and innovation programme under grant agreement No 810573.  D.L. has been supported by the SNSF Ambizione grant $PZ00P2/202123$ hosted at Section de Math\'{e}matiques de l'Universit\'{e} de Gen\`{e}ve, by the Institut de Physique Th\'{e}orique (IPhT) at CEA  Saclay, by the Institut des Hautes \'Etudes Scientifiques (IHES), by the University of Trieste, by the INFN under the national project MMNLP and by the INdAM group GNSAGA. We thank Johannes Schmitt for the invaluable help with the Sage package \textsc{admcycles} and the technical support testing the vanishing of integrals over $\overline{\mathcal{M}}_{g,n}$ on the computational cluster of the MPIM Bonn. G.B. thanks the MPIM Bonn, where part of this work was carried out, for excellent working conditions.

\vspace{1cm}

\section{Background on $\Omega$-classes and topological recursion}
\label{sec:Omega:classes}

\vspace{0.5cm}

\subsection{Definition of the $\Omega$-classes}

In \cite{Mum83}, Mumford derived a formula for the Chern character of the Hodge bundle on the moduli space of curves $\overline{\mathcal{M}}_{g,n}$ in terms of tautological classes and Bernoulli numbers. A generalisation of Mumford's formula was proposed in \cite{Chi08}. The moduli space of stable curves $\overline{\mathcal{M}}_{g,n}$ is substituted by the proper moduli stack $\overline{\mathcal{M}}_{g;a}^{(r,s)}$ of $r$-th roots of the line bundle
\begin{equation}
	\omega_{\log}^{\otimes s}\biggl(-\sum_{i=1}^n a_i p_i \biggr),
\end{equation}
where $\omega_{\log} = \omega\big(\sum_{i = 1}^{n} p_i\big)$ is the log-canonical bundle, $r \in \mathbb{Z}_{> 0}$ and $s, a_1,\ldots,a_n \in \mathbb{Z}$, satisfying the modular constraint
\begin{equation}
\label{eq:acondition}
	a_1 + a_2 + \cdots + a_n = (2g-2+n)s \quad \text{mod}\,\,r.
\end{equation} 
This condition guarantees the existence of a line bundle whose $r$-th tensor power is isomorphic to $\omega_{\log}^{\otimes s}\big(-\sum_{i = 1}^{n} a_i p_i\big)$. The label $a_i$ is called the type of $p_i$. Let $\pi \colon \overline{\mathcal{C}}_{g;a}^{r,s} \to \overline{\mathcal{M}}_{g;a}^{r,s}$ be the universal curve, and $\mathcal{L} \to \overline{\mathcal C}_{g;a}^{r,s}$ the universal $r$-th root. In complete analogy with the case of $\overline{\mathcal{M}}_{g,n}$, one can define $\psi$- and $\kappa$-classes on the moduli spaces of $r$-th roots. There is moreover a natural forgetful morphism
\begin{equation}
	\epsilon \colon
	\overline{\mathcal{M}}^{(r,s)}_{g;a_1,\ldots,a_n}
	\longrightarrow
	\overline{\mathcal{M}}_{g,n}
\end{equation}

Let $B_m(x)$ denote the $m$-th Bernoulli polynomial, that is the polynomial defined by the generating series
\begin{equation}
	\frac{te^{tu}}{ e^t - 1} = \sum_{m \geq 0} B_{m}(u)\frac{t^m}{m!}.
\end{equation}
The evaluations $B_m(0) = (-1)^m B_m(1) = B_m$ recover the usual Bernoulli numbers. There is an explicit formula for the Chern characters of the derived pushforward of the universal line bundle on the moduli of $r$-th roots.

\begin{theorem} \cite{Chi08}
	On the space $\overline{\mathcal{M}}_{g;a_1,\ldots,a_n}^{(r,s)}$ for $a_1,\ldots,a_n \in \{0,\ldots,r - 1\}$, we have the formula
	\begin{equation} \label{eqn:Omega:formula}
			\ch_m\big(R^{\bullet} \pi_{\ast}{\mathcal L} \,\big|\, \overline{\mathcal{M}}_{g;a_1,\ldots,a_n}^{(r,s)}\big)
		=
		\frac{B_{m+1}(\tfrac{s}{r})}{(m+1)!} \kappa_m
		-
		\sum_{i=1}^n \frac{B_{m+1}(\tfrac{a_i}{r})}{(m+1)!} \psi_i^m
		+
		\frac{r}{2} \sum_{a=0}^{r-1} \frac{B_{m+1}(\tfrac{a}{r})}{(m+1)!} \, j_{a\ast} \frac{(\psi')^m - (-\psi'')^m}{\psi' + \psi''}. 
	\end{equation}
	Here $j_a$ is the boundary morphism that represents the boundary divisor such that the two branches of the corresponding node are of type $a$ and $r-a$, and $\psi',\psi''$ are the $\psi$-classes at the two branches of the node.
\end{theorem}

We can then consider the pushforward to the moduli space of stable curves of the family of Chern classes
\begin{equation}\label{eqn:Omega}
	\Omega^{(r,s)}_{g;a_1,\ldots,a_n}(u) 
	=
	\epsilon_{\ast}
	\exp{\Biggl(
		\sum_{m \geq 1} (-u)^m (m-1)! \,\ch_m\big(R^{\bullet} \pi_{\ast}{\mathcal L} \,\big|\, \overline{\mathcal{M}}_{g;a_1,\ldots,a_n}^{(r,s)}\big)	\Biggr)}
	\in
	H^{\text{even}}(\overline{\mathcal{M}}_{g,n}).
\end{equation}
We will omit the variable $u$ when $u = 1$. Notice that we recover Mumford's formula for the Hodge class when $r = s = 1$ and $a = (1,\ldots,1)$. For $r = 1$, general $s$ and $a = (s,\ldots,s)$, we get the generalised Hodge classes considered by Bini in \cite{Bin03}.  If the modular condition \eqref{eq:acondition} is not satisfied, we declare $\Omega^{(r,s)}_{g;a_1,\ldots,a_n}$ to be zero.

\medskip

By expanding the exponential \eqref{eqn:Omega} one can derive an expression of the $\Omega$-classes as a sum over decorated stable graphs \cite{JPPZ17}. From there, one can recognise using the Givental group action that the $\Omega$-classes for types $a_1,\ldots,a_n$ in the fundamental range $\{0,\ldots,r-1\}$  form a cohomological field theory.

\begin{theorem}\label{CohFTl} \cite{lew-pop-sha-zvo17} Let $r \geq 1$ and $V = \text{span}_{\mathbb{C}}(v_1, \ldots, v_r)$ a $r$-dimensional vector space. For any $s \in \ZZ$, the collection of maps $\Omega_{g,n}^{(r,s)} : V^{\otimes n} \longrightarrow H^{\textup{even}}(\overline{\mathcal{M}}_{g,n})$ defined by
$$ 
\Omega^{(r,s)}_{g,n}(v_{a_1}\otimes \cdots \otimes v_{a_n}) = \Omega^{(r,s)}_{g;a_1,\ldots,a_n}
$$
and indexed by $2g - 2 + n > 0$ form a cohomological field theory with the pairing on $V$ defined by 
	\begin{equation}
		\eta(v_a, v_b) = \frac{\delta_{r|a+b}}{r}		
	\end{equation}
If $s \in \{0,\ldots,r\}$, it admits the flat unit $v_s$ (with the convention $v_0 = v_r$). 
\end{theorem}

In fact, we will only need $a_i$ in the fundamental range, but as $a_i = r$  leads to the same class as $a_i = 0$, for convenience we will sometimes allow $a_i = r$ as well or replace the fundamental range with $\{1,\ldots,r\}$. We refer to \cite{JPPZ17,lew-pop-sha-zvo17,GLN20} for a more detailed review of properties and applications of the $\Omega$-classes. Here, we will focus on reviewing its apparition in Hurwitz theory, which is directly relevant for us. Some other basic properties are given in Appendix~\ref{sec:Omega:properties}.

\vspace{0.2cm}

\subsection{Quick definition of topological recursion}
\label{TRreview}
The definition of topological recursion which we give now will only be relevant in Section~\ref{sec:polynom3}. In the rest of the article, topological recursion can be used as a black box which takes a spectral curve as input and gives a collection of multidifferentials $\omega_{g,n}$ indexed by $(g,n) \in \mathbb{Z}_{\geq 0} \times \mathbb{Z}_{> 0}$ as output. The reader ready to accept this may safely jump to the next subsection.

\medskip

We will restrict our definitions to the spectral curves of the particular type needed in this article, and refer \textit{e.g.} to \cite{eyn-ora-07,BEthink,TRLimits} for more general definitions. For us, a spectral curve $\mathscr{S}$ is simply specified by the data of two functions $x,y $ on $\mathbb{C}^*$ such that $\dd x$ and $\dd y$ extend as meromorphic $1$-form on the Riemann sphere containing $\mathbb{C}^*$. We denote $\mathcal{R}$ the set of zeros of $\dd x$. We assume that points in $\mathcal{R}$ are simple zeros of $\dd x$ and are neither zeros nor poles of $\dd y$. 

\medskip 

For each $\alpha \in \mathcal{R}$, there we have a holomorphic involution $z \mapsto \overline{z}$ defined in a neighborhood of $\alpha$ in $\mathbb{C}^*$, such that $x(z) = x(\overline{z})$ but $\overline{z} \neq z$ for $z \neq \alpha$.  The recursion kernel is defined as
$$
K_{\alpha}(z_0,z) = \frac{\frac{1}{2}\big(\frac{1}{z_0 - z} - \frac{1}{z_0 - \overline{z}}\big)\dd z_0}{(y(z) - y(\overline{z}))\dd x(z)}.
$$

\medskip

The topological recursion starts from this data. It then defines the $1$-form and bidifferential
$$
\omega_{0,1}(z) = y(z) \dd x(z), \qquad  \omega_{0,2}(z_1,z_2) = \frac{\dd z_1 \otimes \dd z_2}{(z_1 - z_2)^2},
$$
and by induction on $2g - 2 + (1 + n) > 0$ the multidifferentials
\begin{equation}
\label{Trformula}
\begin{split}
& \quad \omega_{g,1 + n}(z_0,\ldots,z_n) \\
& = \sum_{\alpha \in \mathcal{R}}  \Res_{z = \alpha} K_{\alpha}(z_0,z)\bigg(\omega_{g - 1,2+n}(z,\overline{z},z_1,\ldots,z_n) + \sum_{\substack{J \sqcup J' = \{z_1,\ldots,z_n\} \\ h + h' = g}}^{\text{no}\,\,\omega_{0,1}} \omega_{h,1+|J|}(z,J) \otimes \omega_{h',1+|J'|}(\overline{z},J')\bigg).
\end{split}
\end{equation}
We refer to $\omega_{g,n}$ as the correlators associated to the spectral curve $\mathscr{S}$. We also introduce the corresponding free energies
$$
F_{g,n}(z_1,\ldots,z_n) = \int_{0}^{z_1} \cdots \int_{0}^{z_n} \omega_{g,n}.
$$

\medskip

The technical assumption that $\dd x$ and $\dd y$ do not have common zeros is necessary for topological recursion to be well-defined, \textit{i.e.} to produce symmetric correlators under exchange of $z_1,\ldots,z_n$. The other assumptions can be waived but this will not be needed here.

\vspace{0.2cm}

\subsection{$\Omega$-integrals and topological recursion}
\label{Secrschi}

The only facts we genuinely need about topological recursion are summarised in the remaining of Section~\ref{sec:Omega:classes}. An important fact established in \cite{dun-ora-sha-spi14} is that correlation functions of semi-simple cohomological field theories are governed by topological recursion on a spectral curve specified by the Givental--Teleman reconstruction procedure \cite{Teleman}. This is at the origin of the following result.

\medskip

Let $r \in \mathbb{Z}_{> 0}$ and $s \in \mathbb{Z}^*$, consider the spectral curve $\mathscr{S}^{(r,s)}$ parametrised by $z \in \mathbb{C}^*$ with
\begin{equation}
\label{Srscurve}
\mathscr{S}^{(r,s)}\,: \qquad x(z) = \ln z - z^{r},\qquad y(z) = z^{s},
\end{equation}
and denote $\omega_{g,n}^{(r,s)}$ the  corresponding correlators obtained by the topological recursion. We shall use the symbol $\approx$ to denote an all-order series expansion of a meromorphic form around a certain point.

\begin{theorem} \cite{lew-pop-sha-zvo17}
\label{rsELSV} For $g \geq 0$ and $n \geq 1$ such that $2g - 2 + n > 0$, we have the expansion as $z_i \rightarrow 0$
\begin{equation}
\label{Hgnrs} \omega_{g,n}^{(r,s)}(z_1,\ldots,z_n) \approx \sum_{\mu_1,\ldots,\mu_n > 0} H_{g,n}^{(r,s)}(\mu_1,\ldots,\mu_n) \bigotimes_{i = 1}^n \dd\big(e^{\mu_i x(z_i)},\big),
\end{equation}
where
\[
H_{g,n}^{(r,s)}(\mu_1,\ldots,\mu_n) = \prod_{i = 1}^n \frac{(\mu_i/r)^{\lfloor \mu_i/r \rfloor}}{\lfloor \mu_i/r \rfloor !} \cdot \frac{r^{(2g - 2 + n)(1 + s/r) + |\mu|/r}}{s^{2g -2 + n}} \int_{\overline{\mathcal{M}}_{g,n}} \frac{\Omega^{(r,s)}_{g;-\overline{\mu}}}{\prod_{i = 1}^n \big(1 - \frac{\mu_i}{r}\psi_i\big)}.
\]
\end{theorem}
For $a \in \{1,\ldots,r\}$, let us introduce the functions $\phi_{-1}^{a}(z) = z^{a}$, and for $m \geq 0$ inductively we set
\begin{equation}
\label{phibasis00} \phi_{m}^{a}(z) = \partial_{x(z)} \phi_{m - 1}^{a}(z) = \frac{z}{1 - rz^{r}}\,\partial_{z} \phi_{m - 1}^{a}(z).
\end{equation}
As $z \rightarrow 0$ we have the expansion
\begin{equation}
\label{expfun} a^{-1}\phi_{m}^{a}(z) \approx \sum_{j \geq 0} \frac{(jr + a)^{j}}{j!}\,e^{(jr + a)x(z)}
\end{equation}
The theorem above can be reformulated in terms of the free energies as
\begin{equation}
\label{Equation3} 
\begin{split}
F_{g,n}^{(r,s)}(z_1,\ldots,z_n) & := \int_{0}^{z_1} \cdots \int_{0}^{z_n} \omega_{g,n}(z_1,\ldots,z_n) \\
& = \sum_{\substack{m_1,\ldots,m_n \geq 0 \\ 1 \leq a_1,\ldots,a_n \leq r}} \frac{r^{(2g - 2 + n)(1 + s/r) + \sum_{i = 1}^n a_i/r}}{s^{2g - 2 + n}} \bigg(\int_{\overline{\mathcal{M}}_{g,n}} \Omega^{(r,s)}_{g;r - a} \prod_{i = 1}^n (\psi_i/r)^{d_i} \bigg) \prod_{i = 1}^n a_i^{-1}\phi_{m_i}^{a_i}(z_i).
\end{split}
\end{equation}

\vspace{0.2cm}

\subsection{Weighted Hurwitz numbers and topological recursion}
\label{sec:wHn}

Let 
$$
\hat \psi(\hslash^2, y)\qquad  \text{and} \qquad \hat{y}(\hslash^2,z) = \sum_{m \geq 1} \hat y_m(\hslash^2) z^m
$$
be two bivariate formal power series, with $\hat{\psi}(\hslash^2,0) = 0$. Following \cite{bych-dun-kaz-sha}  we introduce the partition function
\begin{equation}\label{Zpsiy}
Z_{\hat{\psi},\hat{y}} = \sum_{\lambda} s_\lambda\bigg(\frac {p_1}{\hslash},\frac{p_2}{\hslash},\ldots\bigg) s_\lambda\bigg(\frac {\hat y_1(\hslash^2)}{\hslash},\frac{\hat{y}_2(\hslash^2)}{\hslash},\ldots\bigg) \exp{\left(  \sum_{(i,j) \in \lambda} \hat \psi\big(\hslash^2, \hslash(i-j)\big)  \right)}.
\end{equation}
Here, $p_1,p_2,\ldots$ are the power-sum generators of the ring of symmetric polynomials, and $s_{\lambda}$ is the Schur basis, indexed by partitions $\lambda$; $(i,j) \in \lambda$ means that $i  \in \{1,\ldots,\ell(\lambda)\}$ and $j \leq \{1,\ldots,\lambda_i\}$. This $Z_{\hat \psi,\hat y}$ is a hypergeometric tau-function of KP hierarchy with respect to the times $(p_k/k)_{k > 0}$ in the sense of Harnad and Orlov  and admits an interpretation in terms of weighted double Hurwitz numbers \cite{HarnadOrlov,HarnadGuay}. More specifically, it encodes the weighted enumeration of (possibly disconnected) branched covers of $\mathbb{P}^1$ with ramification profile above $0$ tracked by the $p$-variables, ramification profile above $\infty$ tracked by the $\hat{y}$-variables,  type of ramifications elsewhere specified by the weight generating series $\psi$, and topology tracked by the $\hslash$-variable. 
We are interested in the coefficients of expansion
$$
WH^{(\hat{\psi},\hat{y})}_{g,n}(\mu_1,\ldots,\mu_n) = \big[p_{\mu_1} \cdots p_{\mu_n} \hslash^{2g - 2}\big]\,\,\ln Z_{\hat{\psi},\hat{y}},
$$
which restricts the enumeration to connected covers. In \cite{bych-dun-kaz-sha} the representation of $Z_{\hat{\psi},\hat{y}}$ as expectation values in the semi-infinite wedge is the starting point of a detailed analysis of the structural properties of $WH_{g,n}^{(\hat{\psi},\hat{y})}$ which led to topological recursion results in a rather general form. Similar results had been established previously in \cite{ACEH3} by different methods and for a more restricted class of $\hslash$-independent $\hat{\psi}$ and $\hat{y}$.

\medskip
For our purposes, it is sufficient to summarise these results for the specific family of weights
\begin{equation}
\label{PQpsi}
\hat{\psi}(\hslash^2,y) = S(\hslash \partial_y)P(y),\qquad \hat{y}(\hslash^2,z) = Q(z) \qquad \text{with} \qquad S(z) = \frac{\text{sinh}(z/2)}{z/2} = 1 + O(z^2),
\end{equation}
and where $P,Q$ are two polynomials. In this setting consider the spectral curve parametrised by $z \in \mathbb{C}^*$
\begin{equation}
\label{Shatpsi}
\mathscr{S}^{(\hat{\psi},\hat{y})}\,: \qquad x(z) = \ln z - \hat{\psi}(0,y(z)),\qquad y(z) = \hat{y}(0,z),
\end{equation}
and denote $\omega_{g,n}^{(\hat{\psi},\hat{y})}$ the corresponding correlators of the topological recursion. Notice that the spectral curve does not depend on the parameter $\hslash$, but the weight $\hat{\psi}$ does.

\begin{theorem} \cite{bych-dun-kaz-sha} \label{thm:WHTR} For $g \geq 0$ and $n \geq 1$ such that $2g - 2 + n > 0$, for the weights of the form \eqref{PQpsi}, we have the expansion as $z_i \rightarrow 0$:
$$
\omega_{g,n}^{(\hat{\psi},\hat{y})}(z_1,\ldots,z_n) \approx \sum_{\mu_1,\ldots,\mu_n > 0} WH_{g,n}^{(\hat{\psi},\hat{y})}(\mu_1,\ldots,\mu_n) \bigotimes_{i = 1}^{n} \dd \big(e^{\mu_i x(z_i)}\big),
$$
\end{theorem}

The spectral curve $\mathscr{S}^{(ds,s)}$ of \eqref{Srscurve} is obtained by specialising this family to $Q(z) = z^s$ and $P(z) = z^{d}$. In this case, $WH$ corresponds to the $s$-orbifold Hurwitz numbers with $(d + 1)$-completed cycles and Theorem~\ref{thm:WHTR} was established in \cite{bor-kra-lew-pop-sha20,dun-kra-pop-sha19-2}. Together with Theorem~\ref{rsELSV} it gave an ELSV-like formula for those Hurwitz numbers in terms of classes $\Omega^{(ds,s)}$.

\vspace{0.2cm}

\subsection{Principles of deformation}

\label{sec:principle}

We now describe the main principle exploited in this article, namely the behavior of topological recursion under deformations of spectral curves. We will restrict ourselves to spectral curves as defined in Section~\ref{TRreview}. In particular, they are always equipped with the standard fundamental bidifferential
\begin{equation}
\label{om02std}
\omega_{0,2}(z_1,z_2) = \frac{\dd z_1 \otimes \dd z_2}{(z_1 - z_2)^2}.
\end{equation}
and we will not mention it anymore.

Let $\mathscr{S}_t$ be a spectral curve depending analytically on a parameter $t \in \mathcal{T} \subset \mathbb{C}$, where $\mathcal{T}$ is a neighborhood of the closed unit disk. In other words, we are given $x_t(z),y_t(z)$ such that $x_t'(z)$ and $y_t'(z)$ are rational functions of $z$ that depend analytically on $t$. 
 Then, we denote by $\omega_{g,n}^{t}(z_1,\ldots,z_n)$ the correlators of the topological recursion on $\mathscr{S}_t$, and by
$$
F_{g,n}^{t}(z_1,\ldots,z_n) = \int_{0}^{z_1} \cdots \int_{0}^{z_n} \omega_{g,n}^{t}
$$
the corresponding free energies. We assume that the zeros of $\dd x$ in $\mathbb{C}^*$ are simple for all $t \in \mathcal{T}$. Then, the correlators and the free energies are analytic functions of $t \in \mathcal{T}$ --- this can be seen directly from the definition of topological recursion, see also \cite{TRLimits} for a thorough discussion of analyticity in a more general context.

\begin{definition}
In presence of ambiguity, we keep the notation $\partial_t$ the $t$-derivative at fixed $z$ and rather use $D_t$ for the $t$-derivative at $x_t(z)$ fixed.
\end{definition}
We assume that we can represent
\begin{equation}
\label{Dtydx}
\eta_t(z) := D_t\big(y_t(z) \dd x_t(z)\big) = \big(\partial_t y_t(z)\big) \dd x_t(z) - \big(\partial_t x_t(z)\big) \dd y_t(z) = - \Res_{w = \infty} \omega_{0,2}(z,w) f_t(w)
\end{equation}
for some rational function $f_t(z)$ without poles at the zeroes of $x_t'$, where the $t$-derivative in the left-hand side is computed for fixed $z$. Then, Eynard and Orantin have proved \cite[Theorem 5.1]{eyn-ora-07} that
$$
D_t \omega_{g,n}^{t}(z_1,\ldots,z_n) = - \Res_{z = \infty} \omega_{g,n + 1}^{t}(z_1,\ldots,z_n,z)\,f_t(z),
$$
where $D_t$ is the $t$-derivative for $x_t(z_1),\ldots,x_t(z_n)$ kept fixed. Equivalently, the first variation of the free energy is
\begin{equation}
\label{Dtfgn}
D_t F_{g,n}^t(z_1,\ldots,z_n) = \Res_{z = \infty} F_{g,n + 1}^{t}(z_1,\ldots,z_n,w)\,\dd f_t(z) .
\end{equation}

\medskip

We want to use this to compute $F_{g,n}^{t = 1}$ as a Taylor series
$$
F_{g,n}^{1}(z_1',\ldots,z_n') = \sum_{k = 0}^{\infty} \frac{1}{k!} D_{t}^{k} F_{g,n}^{t}(z_1,\ldots,z_n)\big|_{t = 0}
$$
Here one should keep $x_1(z_i') = x_t(z_i)$ all the way before setting $t = 0$. This requires applying repeatedly $D_t$ to \eqref{Dtfgn}. A subtle point is that $f_t(z)$ at $x_t(z)$ may still depend on $t$.

\begin{definition} \label{defclass} We say that a deformation is \emph{horizontal} if $D_t f_t(z) = 0$, where the $t$-derivative is computed at $x_t(z)$ fixed. Otherwise, we say that the deformation is \emph{oblique}.
\end{definition}

For an horizontal deformation, we simply have
$$
F_{g,n}^{1}(z_1',\ldots,z_n') = \sum_{k = 0}^{\infty} \frac{1}{k!} \Res_{w_1 = \infty} \cdots \Res_{w_k = \infty} F_{g,n + k}^{0}(z_1,\ldots,z_n,w_1,\ldots,w_k) \dd f_0(w_1) \otimes \cdots \otimes \dd f_0(w_k),
$$
with $x_1(z_i') = x_0(z_i)$.  For an oblique deformation, the combinatorics of higher derivatives is more involved and leads to \cite[Section 5.3]{BDKLM}
\begin{equation}
\label{Taylorser} F_{g,n}^{1}(z_1',\ldots,z_n') = \sum_{k = 0}^{\infty} \frac{1}{k!} \sum_{l_1,\ldots,l_k = 0}^{\infty}  \Res_{w_1 = \infty} \cdots \Res_{w_k = \infty} F_{g,n + k}^{0}(z_1,\ldots,z_n,w_1,\ldots,w_k) \bigotimes_{c = 1}^{k} \dd\bigg(\frac{D_t^{l_c}f_t(w_c)}{(l_c + 1)!}\bigg)\bigg|_{t = 0}, 
\end{equation}
where again $x_1(z_i') = x_0(z_i)$.

\medskip

The enumerative information in the free energy is typically stored in its decomposition on a suitable basis of functions, or equivalently in its series expansion near a certain point (for us, $z_i = 0$) using the variable $x_t(z_i)$.  For instance, for the case of the spectral curve $\mathscr{S}^{(r,s)}$ of Section~\ref{Secrschi}, this is achieved by the basis
$$
\forall (a,m) \in \{1,\ldots,r\} \times \mathbb{Z}_{\geq 0}
\qquad a^{-1}\phi_m^{a}(z) = \partial_{x(z)}^{m + 1}\big(z^a\big)
$$
through \eqref{Equation3}, or the series expansion \eqref{Hgnrs}. The equivalence between the two came from the series expansion of $\phi_{m}^{a}(z)$ as $z \rightarrow 0$ given in \eqref{expfun}. In the next three sections, we are going to study three families of  spectral curves which all fit the previous setting. For each of them, we will
\begin{itemize}
\item compute the deformation $1$-form \eqref{Dtydx}, \textit{i.e.} find $f_t(z)$. In all three cases, It describes in fact an oblique deformation;
\item decompose $F^0_{g,n + k}$ on a good basis of functions $a^{-1}\phi_{m}^{a}$, so as to express the right-hand side of \eqref{Taylorser} solely in terms of the corresponding coefficients;
\item evaluate the residue pairings in \eqref{Taylorser}, which amounts to compute
$$
T_{m}^{a,(l)} := \Res_{z = \infty} a^{-1}\phi_{m}^{a}(z) \dd\big(D_t^{l}f_t(z)\big)\big|_{t = 0}.
$$
\item expand $F_{g,n}^{1}(z_1',\ldots,z_n')$ of \eqref{Taylorser} as $z_i \rightarrow 0$, using the variables $x_1(z_i') = x_0(z_i)$. This only requires knowing the expansion of $a^{-1}\phi_{m}^{a}(z)$ as $z \rightarrow 0$ in the variable $e^{x_1(z')} = e^{x_0(z)}$.
\end{itemize}
These steps allow to express the $t = 1$ enumerative information as a Taylor series involving only the $t = 0$ enumerative information.  

\vspace{1cm}

\section{The first deformation}
\label{sec:Deformation:one}

\vspace{0.5cm}

\subsection{Setting}

Let $d \geq 1$, $s \geq 2$ and set $r = ds$. Introduce a polynomial $Q(z) = \sum_{j = 1}^{s} q_j\,z^{j}$ and $\sigma \in \mathbb{C}$ such that $q_{s} \sigma \neq 0$, and consider the spectral curve with
\begin{equation}
\label{Dsp1} \mathscr{S}_{1}^{\text{1st}}\,: \quad \left\{\begin{array}{lll} x(z) & = & \ln z - \sigma \big(Q(z)\big)^{d} \\[2pt] y(z) & = & Q(z) \end{array}\right.
\end{equation}
We see it as the $t = 1$ specialisation of the family of spectral curves
$$
\mathscr{S}_t^{\text{1st}} \,: \quad \left\{\begin{array}{lll} x_t(z) & = & \ln z - \sigma\big(q_s z^s + t \tilde{Q}(z)\big)^{d} \\[2pt] y_t(z) & = & q_sz^s + t\tilde{Q}(z) \end{array}\right.
$$
with $\tilde{Q}(z) = \sum_{j = 1}^{s -1} q_jz^j$. At $t = 0$ we get the spectral curve
$$
\mathscr{S}_0^{\text{1st}} \,: \quad \left\{\begin{array}{lll} x_0(z) & = & \ln z - \sigma q_s^{d} z^r  \\[2pt] y_0(z) & = &  q_sz^s \end{array}\right.
$$
If $\sigma$ is chosen small enough relative to $q_1,\ldots,q_s$, the zeros of $x_t'(z)$ remain simple so that the assumptions of Section~\ref{sec:principle} are satisfied. We call $F_{g,n}^{t}$ the free energies associated to $\mathscr{S}_t^{\text{1st}}$.

\medskip

The spectral curve $\mathscr{S}_0^{\text{1st}}$ is related to $\mathscr{S}^{(r,s)}$ of Section~\ref{Secrschi}, which we here denote $\mathscr{S}_{[0]}$:
$$
\mathscr{S}_{[0]} \,:\quad \left\{\begin{array}{lll} x_{[0]}\!(z) & = & \ln z - z^r \\[2pt] y_{[0]}\!(z) & = & z^{s} \end{array}\right.
$$
Indeed, if we use $\tilde{z} = (q_s^{1/s}\sigma^{1/r})z$, we have
$$
x_0(z) = x_{[0]}\!(\tilde{z}) - \frac{\ln q_s}{s} - \frac{\ln \sigma}{r},\qquad y_0(z) = \sigma^{-s/r} y_{[0]}\!(z),
$$
while the standard $\omega_{0,2}$ in \eqref{om02std} has the same expression in the $z$ or the $\tilde{z}$-variables. The homogeneity properties of topological recursion under rescaling imply
\begin{equation}
\label{Fgn000} F_{g,n}^{0}(z_1,\ldots,z_n) = \sigma^{(2g - 2 + n)s/r} \cdot F_{g,n}^{[0]}(\tilde{z}_1,\ldots,\tilde{z}_n).
\end{equation}
The basis of functions we want to use to decompose the free energies is $\phi_{m}^{a,[0]}(\tilde{z})$ from \eqref{phibasis00}, namely
\begin{equation}
\label{phihat0}
\forall (a,m) \in \{1,\ldots,r\} \times \mathbb{Z}_{\geq 0}\qquad a^{-1}\phi_{m}^{a,[0]}(\tilde{z}) = \partial_{x_{[0]}\!(\tilde{z})}^{m + 1} \big(\tilde{z}^{a}\big).
\end{equation}
Equation \eqref{expfun} gives its expansion as $z \rightarrow 0$ in the variable $e^{x_{[0]}\!(\tilde{z})}$, and its expansion in the new variable $e^{x_0(z)} = q_s^{-1/s}\sigma^{-1/r} e^{x_{[0]}\!(\tilde{z})}$ reads:
\begin{equation}
\label{aphi0ep}
a^{-1}\phi_m^{a,[0]}(\tilde{z}) \,\,\mathop{\approx}_{z \rightarrow 0} \,\, \sum_{j \geq 0}  q_s^{(jr + a)/s} \sigma^{j + a/r}\,\frac{(jr + a)^{j}}{j!} e^{(jr + a)x_0(z)}.
\end{equation}

\vspace{0.2cm}

\subsection{The deformation $1$-form} We first compute the deformation $1$-form
$$
\eta_t(z) =  \big(\partial_t y_t(z) \dd x_t(z)\big) - \big(\partial_t x_t(z)\big) \dd y_t(z) = \frac{\tilde{Q}(z)}{z} \dd z.
$$
We recognise that it takes the form
$$
\eta_t(z) = - \Res_{w = \infty} \omega_{0,2}(z,w) f_t(w) \qquad \text{with}\qquad f_t(z) = \sum_{j = 1}^{s - 1} \frac{q_j}{j}\,z^j.
$$ 
Then, we want to evaluate
$$
T_{m}^{a,(l)} = \Res_{z = \infty} a^{-1}\phi_{m}^{a,[0]}(\tilde{z}) \dd \big(D_t^{l} f_t(z)\big)\big|_{t = 0}.
$$

The case $d = 1$ corresponds to $r = s$ and was treated in \cite{BDKLM}. As the function $f_t(z)$ does not depend on $d$, it is straightforward to adapt \cite[Lemma 5.2]{BDKLM} and obtain the following expression.

\begin{lemma}\label{QTev}
For any $l \geq 0$, we have as $z \rightarrow \infty$
$$
D_t^{l}f_t(z)\big|_{t = 0} = \sum_{j = 1}^{s - 1} \frac{\mathcal{Q}_j^{(l)}}{s - j} z^{s - j} + O(1),
$$
where for $j \in \{1,\ldots,s-1\}$ we set
\begin{equation}
\label{Qsum} \frac{\mathcal{Q}_j^{(l)}}{(l + 1)!} = \frac{(-1)^l}{q_s^l} \sum_{\substack{\rho \in \mathcal{P}_{s - 1} \\ \ell(\rho) = l + 1 \\ |\rho| = j}} \bigg[\frac{s - |\rho|}{r}\bigg]_{\ell(\rho) - 1}\cdot \frac{\vec{q}_{s - \rho}}{|\text{Aut}(\rho)|},
\end{equation}
Besides, 
\begin{equation}
\label{Tmal1} T_{m}^{a,(l)} = \delta_{m,0} \cdot \left\{\begin{array}{lll} r^{-1}\sigma^{a/r - 1}q_s^{(a - r)/s} \mathcal{Q}_{s -r + a}^{(l)} & & \text{if} \,\,a \in \{r - s + 1,\ldots,r - 1\} \\[2pt] 0 & & \text{otherwise} \end{array}\right.
\end{equation}
\end{lemma} 
The conditions in the sum \eqref{Qsum} force $\rho$ to be non-empty. Note that $d = 1$ corresponds to 

\vspace{0.2cm}

\subsection{The Taylor series}

We now have all the ingredients to evaluate the Taylor series \eqref{Taylorser}:
\begin{itemize}
\item The free energies $F_{g,n + k}^{0}$ are equal to the $F_{g,n + k}^{[0]}$ up to a rescaling given by \eqref{Fgn000};
\item The decomposition of $F_{g,n + k}^{[0]}$ on the basis functions $a^{-1}\phi_{m}^{a,[0]}$ is given in \eqref{Equation3} in terms of intersection indices of the classes $\Omega^{(ds,s)}$;
\item The residue pairing of the basis functions with $\dd(D_t^{l} f_t)\big|_{t = 0}$ is given by \eqref{Tmal1};
\end{itemize}
Substituting these information in the Taylor series \eqref{Taylorser} we arrive to
\begin{equation}
\label{equation8} \begin{split}
 F_{g,n}^{1}(z_1',\ldots,z_n') & = \bigg(\frac{\sigma^{s/r}r^{1 + s/r}}{s}\bigg)^{2g - 2 + n}  \sum_{\substack{1 \leq a_1,\ldots,a_n \leq r \\ m_1,\ldots,m_n \geq 0}} r^{\sum_{i = 1}^n a_i/r} \prod_{i = 1}^n a_i^{-1}\phi_{m_i}^{a_i,[0]}(\tilde{z}_i) \\
& \quad \times \sum_{k \geq 0} \frac{(s^{-1}r^{s/r})^{k}}{k!} \sum_{\substack{r - s + 1 \leq b_1,\ldots,b_{k} \leq r - 1 \\ l_1,\ldots,l_k \geq 0}} r^{\sum_{c = 1}^{k} b_{c}/r} \sigma^{\sum_{c = 1}^{k} (s - r + b_{c})/r} q_s^{\sum_{c = 1}^{k} (b_{c} - r)/s} \\
& \quad \qquad \qquad \qquad \qquad\qquad \qquad \times \bigg(\int_{\overline{\mathcal{M}}_{g,n + k}} \Omega^{(r,s)}_{g;r - \boldsymbol{a},r -\boldsymbol{b}} \prod_{i = 1}^n (\psi_i/r)^{m_i}\bigg) \prod_{c = 1}^{k} \frac{\mathcal{Q}_{s - r + b_{c}}^{(l_{c})}}{(l_{c} + 1)!}.
\end{split} 
\end{equation} 
where $x_1(z'_i) = x_0(z_i) = q_s^{-1/s}\sigma^{-1/r} x_{[0]}\!(\tilde{z}_i)$. Now, we expand the right-hand side as $z_i \rightarrow 0$ using the variables $e^{x_0(z)}$:
\begin{equation}
\label{DHdef}
F_{g,n}^{1}(z_1',\ldots,z_n') \approx \sum_{\mu_1,\ldots,\mu_n > 0} H_{g,n}^{\text{1st}}(\mu_1,\ldots,\mu_n) \prod_{i = 1}^n e^{\mu_i x_1(z_i')}.
\end{equation}
The computation of the coefficients $H_{g,n}^{\text{1st}}$ is achieved via the expansion of the basis functions given in \eqref{aphi0ep}, and we also insert the expression for $\mathcal{Q}$ provided by Lemma~\ref{QTev}. Recalling $r = ds$, this yields for any partition $\mu$ of length $n > 0$:
\begin{equation} 
\label{DHrs} \begin{split}
H_{g,n}^{\text{1st}}(\mu) & = \frac{r^{(2g - 2 + n)(1 + s/r)}}{s^{2g - 2 + n}} \sum_{\substack{\overline{\lambda} \in \overline{\mathcal{P}}_s \\ |\overline{\lambda}| = |\mu|}}  \sigma^{(2g - 2 + n + \ell(\overline{\lambda}))s/r} \prod_{i = 1}^n \frac{(\mu_i/r)^{\lfloor \mu_i/r \rfloor}}{\lfloor \mu_i/r \rfloor !} \cdot \vec{q}_{\overline{\lambda}}  \\
& \quad \times\left( \sum_{k = 0}^{\ell(\lambda)} \frac{(-1)^{\ell(\lambda) - k} (r/s)^{k}}{k!}  \sum_{\substack{\boldsymbol{\rho} \in \mathcal{P}\!\!\mathcal{P}_{s - 1}^{k} \\ \sqcup \boldsymbol{\rho} = s - \lambda}} \prod_{c = 1}^{k} \frac{\big[\frac{s - |\rho^{(c)}|}{r}\big]_{\ell(\rho^{(c)}) - 1}}{|\text{Aut}(\rho^{(c)})|}  \cdot \bigg(\int_{\overline{\mathcal{M}}_{g,n + k}} \frac{\Omega^{(r,s)}_{g;-\overline{\mu},s - |\boldsymbol{\rho}|}}{\prod_{i = 1}^n \big(1 - \frac{\mu_i}{r}\psi_i\big)}\bigg)\right).
\end{split} 
\end{equation}
The notation for partitions are explained in Section~\ref{Sec:not}. The exponent $k$ in $\mathcal{P}\!\!\mathcal{P}_{s -1}^{k}$ just means that we are considering $k$-tuples of $(s - 1)$-partitions. If $\lambda$ is empty, \textit{i.e.} $\overline{\lambda}$ has only parts $s$, the sum in the last line only has a $k = 0$ term which is equal to $1$; if $\lambda$ is non-empty, the sum in the last line should start at $k = 1$.

\vspace{0.2cm}

\subsection{Polynomiality and vanishing}
\label{sec:polyfirst}
In view of Section~\ref{sec:wHn} and Theorem~\ref{thm:WHTR}, the numbers $H^{\text{1st}}_{g,n}(\mu)$ defined by \eqref{DHdef} coincide with the weighted double Hurwitz numbers $WH^{(\hat{\psi},\hat{y})}_{g,n}(\mu)$ in the special case
\begin{equation*}
\begin{split}
\hat{\psi}(\hslash^2,y) & = S(\hslash \partial_y) y^{d} = \sum_{j = 0}^{\lfloor d/2 \rfloor} \frac{d!}{(d - 2j)!(2j + 1)!}\,(\hslash/2)^{2j} y^{d - 2j}, \\
\hat{y}(\hslash^2,z) & = Q(z) = \sum_{j = 1}^{s} q_j z^j.
\end{split}
\end{equation*}
In particular, for any fixed $\mu$,  $H^{\text{1st}}_{g,n}(\mu)$ must be a polynomial in $q_1,\ldots,q_s$ and cannot contain negative powers of $q_s$. Negative powers of $q_s$ occur when $\lambda$ is a non-empty $(s - 1)$-partition such that $|\lambda| > |\mu|$. In terms of the non-empty $(s - 1)$-partition $\nu = s - \lambda' = \check{\lambda}'$, this condition reads
\begin{equation}
\label{munud}
|\mu| + |\nu| < s\ell(\nu).
\end{equation}
Due to \eqref{eq:acondition} and recalling that $r = ds$, we need $|\mu| + |\nu|$ to be divisible by $s$ for the $\Omega$-class itself to be non-zero. Therefore, the condition can be written in the stronger form
\begin{equation}
\label{munuaa}|\mu| + |\nu| \leq s(\ell(\nu) - 1).
\end{equation}

This gives us the following vanishing of $\Omega$-integrals, valid for any non-empty partition $\mu$ and any $(s - 1)$-partition $\nu$ satisfying \eqref{munuaa}
\begin{equation}
\sum_{k = 1}^{\ell(\nu)} \frac{(-1)^{\ell(\nu) - k} (r/s)^{k}}{k!} 
\sum_{
\substack{\boldsymbol{\rho} \in (\mathcal{P}\!\!\mathcal{P}_{s - 1})^{k} 
\\ 
\sqcup \boldsymbol{\rho} = \nu}
}
\prod_{c = 1}^{k} 
\frac{\big[\frac{s - |\rho^{(c)}|}{s}\big]_{\ell(\rho^{(c)}) - 1}}{|\text{Aut}(\rho^{(c)})|}  \cdot \int_{\overline{\mathcal{M}}_{g,\ell(\mu) + k}} \frac{\Omega^{(r,s)}_{g;-\overline{\mu},s - |\boldsymbol{\rho}|}}{\prod_{i = 1}^{\ell(\mu)} \left(1 - \frac{\mu_i}{r}\psi_i\right)} = 0.
\end{equation}

If furthermore $\nu$ is bounded, \textit{i.e.} $\min_{i \neq j} (\nu_i + \nu_j) \geq s$ for any $i \neq j$, the only way to write $\nu$ as a concatenation of $k$ partitions of size $\leq s - 1$ is to take $k = \ell(\nu)$ and  $\rho^{(1)} = (\nu_1),\ldots,\rho^{(k)} = (\nu_k)$ up to permutation. Therefore, the sum has $k!$ equal terms and we get the vanishing of a single $\Omega$-integral
\begin{equation}
\label{bvanish}
\int_{\overline{\mathcal{M}}_{g,\ell(\mu) + \ell(\nu)}} \frac{\Omega^{(r,s)}_{g;-\overline{\mu},s - \nu}}{\prod_{i = 1}^{\ell(\mu)} \left(1 - \frac{\mu_i}{r}\psi_i\right)} = 0.
\end{equation}
 
This proves Theorem~\ref{th:A}. Notice that in the first deformation $r = ds$. We have checked numerically that such a result --- for instance \eqref{bvanish} --- does not hold if $s$ does not divide $r$.

\vspace{1cm}

\section{The second deformation}
\label{sec:Deformation:two}

\vspace{0.5cm}

\subsection{Setting}
\label{S41}
We set $s \geq 1$, $d \geq 2$ (notice the difference with the assumption in Section~\ref{sec:Deformation:one}) and set $r = ds$. We introduce  two polynomials
$$
P(y) = \sum_{j = 1}^{d} p_j y^{j},\qquad  Q(z) = \sum_{i = 1}^{s} q_i z^{i},
$$
 a parameter $\sigma \in \mathbb{C}$ such that $\sigma q_{s} p_{d} \neq 0$
 and consider the spectral curve
$$
\mathscr{S}_{1}^{\text{2nd}}\,: \quad \left\{\begin{array}{lll} x(z) & = & \ln z - \sigma P( Q(z)) \\[2pt] y(z) & = & Q(z) \end{array}\right.
$$
If $P(y) = y^{d}$ this is the spectral curve $\mathscr{S}_{1}^{\text{1st}}$ of Section~\ref{sec:Deformation:one}. We take a different perspective now as the deformation will consist in turning on the other coefficients of $P$ while keeping $Q$ fixed. Namely, we see $\mathscr{S}_1^{\text{2nd}}$ as the $t = 1$ specialisation of the family of spectral curves
$$
\mathscr{S}_{t}^{\text{2nd}}\,: \quad \left\{\begin{array}{lll} x_t(z) & = & \ln z - \sigma\big(p_d Q(z)^{d} + t\tilde{P}(Q(z))\big) \\[2pt] y_t(z) & = & Q(z) \end{array}\right.
$$
with $\tilde{P}(y) = \sum_{j = 1}^{d - 1} p_jy^{j}$. If $\sigma$ is chosen small enough relative to $q_1,\ldots,q_s$, the zeros of $x_t'(z)$ remain simple so that the assumptions of Section~\ref{sec:principle} are satisfied.  We denote $F_{g,n}^{t}$ the free energies associated to $\mathscr{S}_{t}^{\text{2nd}}$. 

\medskip

At $t = 0$ we have
$$
\mathscr{S}_{0}^{\text{2nd}}\,: \quad \left\{\begin{array}{lll} x_{0}(z) & = &  \ln z - \sigma p_d Q(z)^{d} \\[2pt] y_0(z) &=& Q(z) \end{array}\right.\,.
$$
Up to replacing $\sigma$ with $\sigma p_d$ this is the spectral curve $\mathscr{S}_{1}^{\text{1st}}$ so we have the intersection-theoretic formula \eqref{equation8} for $s \geq 2$. Namely, the correlation functions of $\mathscr{S}_{0}^{\text{2nd}}$ can be written:
\begin{equation}
\label{Equation322}
\begin{split}
& \quad F_{g,n}^{0}(z_1,\ldots,z_n) \\ 
& = \big((\sigma p_d)^{s/r} s^{-1} r^{1 + s/r}\big)^{2g - 2 + n} \sum_{\substack{1 \leq a_1,\ldots,a_n \leq r \\ m_1,\ldots,m_n \geq 0}} r^{\sum_{i = 1}^n a_i/r} \prod_{i = 1}^n a_i^{-1}\phi_{m_i}^{a_i,[0]}(\tilde{z}_i) \\
& \quad \times \sum_{k \geq 0} \frac{(s^{-1} r^{s/r})^{k}}{k!} \sum_{\substack{r-s + 1 \leq b_1,\ldots,b_k \leq r - 1 \\ l_1,\ldots,l_k \geq 0}} r^{\sum_{c = 1}^{k} b_{c}/r} (\sigma p_d)^{\sum_{c = 1}^{k}  (s - r + b_{c})/r} q_s^{\sum_{c = 1}^{k} ( b_{c} - r)/s} \\
& \qquad\qquad\qquad\qquad\qquad\qquad\qquad\qquad\qquad \times \bigg(\int_{\overline{\mathcal{M}}_{g,n + k}} \Omega^{(r,s)}_{g;r - \boldsymbol{a},r - \boldsymbol{b}} \prod_{i = 1}^n (\psi_i/r)^{m_i}\bigg) \prod_{c = 1}^{k} \frac{\mathcal{Q}^{(l_{c})}_{s - r + b_{c}}}{(l_{c} + 1)!},
\end{split}
\end{equation}
where the functions $\phi_{m}^{a,[0]}$ are the ones defined in  \eqref{phihat0}, $\mathcal{Q}$s are taken from Lemma~\ref{QTev}, and
\begin{equation}
\label{tildezzz}
\tilde{z}_i = q_s^{1/s} (\sigma p_d)^{1/r} z_i.
\end{equation}

The case $s = 1$ was excluded in Section~\ref{sec:Deformation:one}. In this case, there is no deformation and we simply have the term $k = 0$ in the sum \eqref{Equation322}, that is
$$
F_{g,n}^{0}(z_1,\ldots,z_n) = \big((\sigma p_d)^{1/r} r^{1 + 1/r}\big)^{2g - 2 + n} \!\!\!\sum_{\substack{1 \leq a_1,\ldots,a_n \leq r \\ m_1,\ldots,m_n \geq 0}} \!\!\!\!\! r^{\sum_{i = 1}^n a_i/r} \prod_{i = 1}^{n} a_i^{-1} \phi_{m_i}^{a_i,[0]}(\tilde{z}_i) \int_{\overline{\mathcal{M}}_{g,n}} \!\!\! \Omega^{(r,1)}_{g;r-\boldsymbol{a}} \prod_{i = 1}^{n} (\psi_i/r)^{m_i},
$$
with $\tilde{z}_i = q_1 (\sigma p_r)^{1/r} z_i$.

 \vspace{0.2cm}

\subsection{The deformation $1$-form}

We compute
$$
\eta_t(z) = \big(\partial_t y_t(z)\big)\dd x(z) - \big(\partial_{t} x_t(z)\big)\dd y(z) = \sigma \tilde {P}( Q(z))\ Q'(z)\dd z
$$
and try to represent it as
$$
\eta_t(z) = - \mathop{{\rm Res}}_{w = \infty} \omega_{0,2}(z,w)\,f_t(w)
$$
We have $\tilde {P}(y) = p_{d - 1}y^{d - 1} + O(y^{d - 2})$ as $y \rightarrow \infty$ and we can find coefficients $(f_j)_{j = 1}^{r - 1}$ which are polynomials in $p_1,\ldots,p_{d - 1}$ and $q_1,\ldots,q_s$ such that
$$
f_t(z) = \sum_{j = 1}^{r-1} \frac{f_j}{j}\,z^j
$$
does the job. We only need the following information on $f_t(z)$.

\begin{lemma}
\label{Fll} We have for $l \geq 0$ as $z \rightarrow \infty$
\[
D_t^l f_t(z)\big|_{t = 0} = \sum_{j = 1}^{r - 1} \frac{\mathcal{F}_{j}^{(l)}\,z^{r - j}}{r - j} + O(1),
\]
where
\begin{equation}
\label{Fjlll}
\frac{\mathcal{F}_{j}^{(l)}}{(l + 1)!} =  \frac{(-1)^l\sigma (r - j)}{p_d^{l}}
 \sum_{\substack{(\theta,\gamma) \in \mathcal{P}_{s - 1} \times \mathcal{P}_{d - 1} \\ |\theta| + s|\gamma| = j + s \\ \ell(\gamma) = l + 1}} \frac{q_s^{d + 1 - |\gamma| - \ell(\theta)} \,(d - |\gamma|)!}{(d + 1 - |\gamma| - \ell(\theta))!}\,\cdot \, \bigg[ \frac{d + 1 - |\gamma|}{d}\bigg]_{l}\,\frac{\vec{p}_{d - \gamma}}{|\text{Aut}(\gamma)|} \cdot \frac{\vec{q}_{s - \theta}}{|\text{Aut}(\theta)|},
\end{equation}
with the convention $p_0 = q_0 = 1$. For $s = 1$, there is a huge simplification as the sum over $\theta$ is absent, and we have for $j \in \{1,\ldots,d - 1\}$
\begin{equation}
\label{Fjsimple}
\frac{\mathcal{F}_{j}^{(l)}}{(l + 1)!} = \frac{(-1)^{l} \sigma}{p_d^{l}} \sum_{\substack{\gamma \in \mathcal{P}_{d - 1} \\ |\gamma| = j + 1 \\ \ell(\gamma) = l + 1}} q_1^{d + 1 - |\gamma|} \cdot \bigg[\frac{d + 1 - |\gamma|}{d}\bigg]_{l} \cdot \frac{\vec{p}_{d - \gamma}}{|\text{Aut}(\gamma)|}.
\end{equation}
\end{lemma}
The conditions on the sum \eqref{Fjlll} force $\gamma$ to have positive length, hence to be non-empty, whereas $\theta$ is allowed to be empty. In particular, in the $s = 1$ case, $\theta = \emptyset$ is the only contribution, and then $\gamma$ is forced to have length $j + 1 \geq 2$. The constraints on the sizes of $\gamma$ and $\theta$ impose that only nonnegative powers for $q_s$ appear in \eqref{Fjlll}, and the factorial in the denominator is then non-zero.

\medskip

Then, we would like to compute
\begin{equation}
\label{Tamal}
T_{m}^{a,(l)} =  \mathop{{\rm Res}}_{z = \infty} a^{-1} \phi_{m}^{a,[0]}(\tilde{z})\dd\big(D_t^{l} f_t(z)\big)\big|_{t = 0}.
\end{equation}
Starting from Lemma~\ref{Fll}, evaluating $T_m^{a,(l)}$ follows the same steps as \cite[Lemma 5.3]{BDKLM} and we only state the result.
\begin{corollary}
\label{Tlemma22} For $a \in \{1,\ldots,r\}$, $m,l \geq 0$, we have
\[
T_{m}^{a,(l)} = \delta_{m,0} \cdot \left\{\begin{array}{lll} r^{-1}\,q_{s}^{a/s - d} (\sigma p_d)^{a/r - 1}\,\mathcal{F}_{a}^{(l)} & & {\rm if}\,\,a \in \{1,\ldots,r - 1\} \\[2pt] 0 & & {\rm if}\,\,a = r \end{array}\right.\,.
\]
\end{corollary}

\begin{proof}[Proof of Lemma~\ref{Fll}]
There exists a unique formal series $z_t = z + \mathbb{K}[[z^{-1}]][[t]]$ satisfying $x_0(z) = x_t(z_t)$, that is
\begin{equation}
\label{lncarac} \ln z_t - \sigma \bigg(p_d \big(Q(z_t)\big)^{d} + t\sum_{j = 1}^{d - 1}p_j  \big(Q(z_t)\big)^{j}\bigg) = \ln z - \sigma p_d \big(Q(z)\big)^{d}.
\end{equation}
We need to compute the generating series
\[
f(z,t) := \sum_{l \geq 0} \big(D_u^{l}f_u(z)\big)\big|_{u = 0}\,\frac{t^{l}}{l!} \,\,\mathop{=}_{z \rightarrow \infty} \,\,\sum_{j = 1}^{r - 1} \frac{\mathcal{F}_j(t)}{r - j}\,z^{r - j} + O(1),
\]
which is such that
\[
\dd_z f(z,t) = \sigma \tilde {P}(Q(z_t)) \dd\big(Q(z_t)\big).
\]
Therefore, we have for $j \in \{1,\ldots,r - 1 \}$
\[
\mathcal{F}_{j}(t)  = - \mathop{{\rm Res}}_{z = \infty} f(z,t)\,(r - j) z^{-(r - j + 1)}\,\dd z = \mathop{{\rm Res}}_{z = \infty} \dd_{z}\big(f(z,t)\big) z^{-(r - j)} = -\mathop{{\rm Res}}_{z = \infty} \sigma \tilde {P}( Q(z_t)) \dd Q(z_t)\,z^{-(r - j)}.
\]
Using the change of variable $\zeta_t =  \big(Q(z_t)\big)^{1/s}$, we get
\[
\mathcal{F}_{j}(t) = - \sum_{c = 1}^{d - 1} \sigma p_{c}\mathop{{\rm Res}}_{\zeta_t = \infty} \zeta_t^{sc}\dd(\zeta_t^{s})\,z^{-(r - j)} \\ = \sum_{c = 1}^{d - 1} \sigma p_c s\,\big[\zeta_t^{-s(c + 1)}\big]\, z^{-(r - j)}.
\]
We first write the series expansion
\[
z^{-(r - j)}  \mathop{\approx}_{\zeta \rightarrow \infty}\,\, \sum_{k \geq 0} A_{j,k}\,\zeta^{-(r - j + k)},
\]
where $\zeta = \big(Q(z)\big)^{1/s}$, and compute its coefficients
\begin{equation}
\label{Ajk} \begin{split} 
A_{j,k} & = -\mathop{{\rm Res}}_{\zeta = \infty} z^{-(r - j)}\zeta^{r - j + k - 1}\dd u  = - \frac{r - j}{r - j + k} \mathop{{\rm Res}}_{z = \infty} z^{-(r - j + 1)} \zeta^{r - j + k}\,\dd z \\
& = \frac{r - j}{r - j + k}\,\big[z^{r - j}\big]\,\zeta^{r - j + k} = \frac{r - j}{r - j + k}\, \big[z^{r - j}\big]\,\big( Q(z)\big)^{(r - j + k)/s} \\
& = \frac{(r - j)q_s^{(r - j + k)/s}}{r - j + k}\,\big[z^{-k}\big] \left(1 + \sum_{a = 1}^{s - 1} \frac{q_a}{q_s} z^{a - s}\right)^{(r - j + k)/s} \\ 
& =  \frac{r - j}{r - j + k} \sum_{\substack{\theta \in \mathcal{P}_{s - 1} \\ |\theta| = k}} q_s^{(r - j  + |\theta| - s \ell(\theta))/s}\,\frac{\big(\frac{r - j + k}{s}\big)!}{\big(\frac{r - j + k}{s} - \ell(\theta)\big)!}\,\frac{\vec{q}_{s - \theta}}{|\text{Aut}(\theta)|}.
\end{split}  
\end{equation} 
Then, we have to compute
\begin{equation}
\label{tildeFjt}
\mathcal{F}_{j}(t) = \sum_{c = 1}^{d - 1} \sum_{k \geq 0} \sigma p_{c} s \cdot A_{j,k} \cdot \big[\zeta_t^{-s(c + 1)}\big]\,\zeta^{-(r  - j + k)}.
\end{equation}
Imagine we knew the power series expansion of $\zeta$ in terms of $\zeta_t$ up to $O(\zeta_t^{-\alpha})$ as $\zeta_t \rightarrow \infty$, for some $\alpha$. Then, we would know the series expansion of  $\zeta^{-(r - j + k)}$ up to $O(\zeta_t^{-(r - j + k + \alpha + 1)})$. To compute \eqref{tildeFjt} for a fixed $j \in \{1,\ldots,r - 1\}$, we therefore need $r - j + k + \alpha + 1 > (c + 1)s$ for all $c \in \{1,\ldots,d - 1\}$ and $k \geq 0$.  This request would be fulfilled with $\alpha \geq j$. Notice that ignoring the logarithm in the characterisation \eqref{lncarac} gives $u$ (and thus $z$) as a series in $\zeta_t$ up to $O(\zeta_t^{-r})$, namely
\[
\zeta = \zeta_t \left(1 + \sum_{j = 1}^{d - 1} \frac{tp_j}{p_d \zeta_t^{(d - j)s}} + O(\zeta_t^{-(r + 1)})\right)^{1/r}.
\]
This truncated characterisation is therefore sufficient to compute all $\mathcal{F}_{j}(t)$ for $j \in \{1,\ldots,r - 1\}$. Let us write for $k \geq 1$
\begin{equation}
\zeta^{-k} \mathop{\approx}_{\zeta_t \rightarrow \infty}\,\, \sum_{m \geq 0} C_{k,m}(t)\,\zeta_t^{-(k + ms)},
\end{equation}
where for $m \leq d$
\begin{equation}
\label{Ckmt} \begin{split}
C_{k,m}(t) & = -\mathop{{\rm Res}}_{\zeta_t = \infty} \zeta^{-k}\,\zeta_t^{k + ms - 1} \dd \zeta_t =  \big[\zeta_t^{-ms}\big] \left(1 + \sum_{j = 1}^{d - 1} \frac{tp_j}{p_{d}\zeta_t^{(d - j)s}}\right)^{-k/r} \\
& = \sum_{\substack{\beta \in \mathcal{P}_{d - 1} \\ |\beta| = m}} \frac{(-1)^{\ell(\beta)} t^{\ell(\beta)}}{p_d^{\ell(\beta)}}\cdot \bigg[\frac{k}{r}\bigg]_{\ell(\beta)} \cdot \frac{\vec{p}_{d - \beta}}{|\text{Aut}(\beta)|},
\end{split}
\end{equation} 
We observe that there are $m$ are integers in the series \eqref{Ckmt}. Therefore:
$$
\mathcal{F}_{j}(t) = \sum_{c = 1}^{d - 1} \sum_{k \in \mathbb{Z}} A_{j,j + ks} \cdot C_{(d + k)s,c + 1 - d - k}(t).
$$
We then substitute the value of $A$ from \eqref{Ajk} and $C(t)$ from \eqref{Ckmt}, and extract from the latter the coefficient of $t^{l}/l!$ and divide by $(l + 1)!$. This yields
\begin{equation*}
\begin{split} 
\frac{\mathcal{F}_{j}^{(l)}}{(l + 1)!} & = \sum_{c = 1}^{d - 1} \sigma p_c s\,\frac{r - j}{r + ks} \sum_{\substack{\theta \in \mathcal{P}_{s - 1} \\ |\theta| = j + ks \\ k \in \mathbb{Z}}} q_s^{d + k - \ell(\theta)} \frac{(d + k)!}{(d + k - \ell(\theta))!}\,\frac{\vec{q}_{s - \theta}}{|\text{Aut}(\theta)|} \\ 
& \quad \qquad\qquad \qquad \times \sum_{\substack{\beta \in \mathcal{P}_{d - 1} \\ |\beta| = c + 1 - (d + k) \\ \ell(\beta) = l}} \frac{(-1)^{l} p_d^{-l}}{l + 1} \cdot \bigg[1 + \frac{k}{d}\bigg]_{l} \cdot \frac{\vec{p}_{d - \beta}}{|\text{Aut}(\beta)|}.
\end{split}
\end{equation*}
In this sum, we can absorb the extra factor $p_c$ by defining a new partition $\gamma$ obtained by adding to $\beta$ a part $d - c$. Then $|\gamma| = |\beta| + d - c = 1 - k$ and $d + k = d + 1 - |\gamma|$. Then $\frac{s}{r + ks} = \frac{1}{d + k}$ can be absorbed in the factorial $(d + k)!$, turning it into $(d + k - 1)! = (d - |\gamma|)!$. Since $\ell(\gamma) = l + 1$, we remark that
$$
\sum_{c} \sum_{\beta} \frac{1}{l + 1}\,\frac{p_c \cdot \vec{p}_{d - \beta}}{|\text{Aut}(\beta)|} \cdots = \sum_{\gamma} \frac{\vec{p}_{d - \gamma}}{|\text{Aut}(\gamma)|} \cdots 
$$
Besides, we see that $|\theta| = j + ks = j + (1 - |\gamma|)s$, thus $|\theta| + s|\gamma| = j + s$. These handlings yield the claimed formula \eqref{Fjlll}. For $s = 1$ the sum over $\theta$ is absent, we have $|\gamma| = j + 1$  and $r = d$, and the  simplification
$$
(r - j) \cdot \frac{(d - |\gamma|)!}{(d + 1 - |\gamma|)!} \cdot \bigg[\frac{d + 1 - |\gamma|}{d}\bigg]_{l} = (d - j) \cdot \frac{1}{d - j} \cdot \bigg[1 - \frac{j}{d}\bigg]_{l} = \bigg[1 - \frac{j}{d}\bigg]_{l}
$$
leads to the claimed \eqref{Fjsimple}. 
\end{proof}

\vspace{0.2cm}

\subsection{The Taylor series}

We have all the ingredients to evaluate Taylor series \eqref{Taylorser}:
\begin{itemize}
\item The decomposing of the free energies $F_{g,n + k}^{0}$ on the basis functions  $a^{-1}\phi_{m}^{a,[0]}$ is given in \eqref{Equation322} in terms of intersection indices of the classes $\Omega^{(ds,s)}$;
\item The residue pairing of the basis functions with $\dd(D_t^{l} f_t)\big|_{t = 0}$ was found in Corollary~\ref{Tlemma22}.
\end{itemize}
This leads to:
\begin{equation}
\begin{split} 
& \quad F_{g,n}^{1}(z_1',\ldots,z'_n) \\ 
& = \big((\sigma p_d)^{s/r} s^{-1} r^{1 + s/r}\big)^{2g - 2 + n} \sum_{\substack{1 \leq a_1,\ldots,a_n \leq r \\ m_1,\ldots,m_n \geq 0}} r^{\sum_{i = 1}^{n} a_i/r} \prod_{i = 1}^n a_i^{-1}\phi_{m_i}^{a_i,[0]}(\tilde{z}_i) \\
& \quad \times \sum_{k \geq 0} \frac{(s^{-1}r^{s/r})^{k}}{k!} \sum_{\substack{r - s + 1 \leq b_1,\ldots,b_k \leq r - 1 \\ l_1,\ldots,l_k \geq 0}} r^{\sum_{c = 1}^{k} b_{c}/r} (\sigma p_d)^{\sum_{c = 1}^{k} (s - r + b_{c})/r} q_s^{\sum_{c = 1}^{k} (b_{c} - r)/s} \\ 
& \quad \times \sum_{h \geq 0} \frac{(s^{-1}r^{s/r})^{h}}{h!} \sum_{\substack{1 \leq j_1,\ldots,j_{h} \leq r - 1 \\ o_1,\ldots,o_{h} \geq 0}} r^{\sum_{c = 1}^{h} j_{c}/r} (\sigma p_d)^{\sum_{c = 1}^{h} (s - r + j_{c})/r} q_s^{\sum_{c = 1}^{h} (j_{c} - r)/s} \\  
&  \qquad\qquad \qquad \qquad  \times \bigg(\int_{\overline{\mathcal{M}}_{g,n + k + h}} \Omega^{(r,s)}_{g;r - \boldsymbol{a},r - \boldsymbol{b},r - \boldsymbol{j}} \prod_{i = 1}^n (\psi_i/r)^{m_i}\bigg)   \prod_{c = 1}^{k} \frac{\mathcal{Q}_{s - r + b_{c}}^{(l_{c})}}{(l_{c} + 1)!} \prod_{c = 1}^{h} \frac{\mathcal{F}_{j_{c}}^{(o_{c})}}{(o_{c} + 1)!},
\end{split}  
\end{equation}
with $x_1(z'_i) = x_0(z_i)$. Taking into account the series expansion \eqref{expfun} for the basis functions and $e^{x_1(z'_i)} = e^{x_0(z_i)} = q_s^{-1/s}(\sigma p_d)^{-1/ds}e^{x_{[0]}\!(\tilde{z}_i)}$, we deduce the all-order series expansion as $z_i' \rightarrow 0$
\[
F_{g,n}^{1}(z_1',\ldots,z'_n) \approx \sum_{\mu_1,\ldots,\mu_n > 0} H_{g,n}^{\text{2nd}}(\mu_1,\ldots,\mu_n) \prod_{i = 1}^n e^{\mu_i x_1(z_i')},
\]
with
\begin{equation}
\label{Hgn2nd} \begin{split}
H_{g,n}^{\text{2nd}}(\mu) & = \big((\sigma p_d)^{s/r} s^{-1}r^{1 + s/r}\big)^{2g - 2 + n} \prod_{i = 1}^n \frac{(\mu_i/r)^{\lfloor \mu_i/r \rfloor}}{\lfloor \mu_i/r\rfloor !} \sum_{k,h \geq 0} \,\,\sum_{\substack{r - s + 1 \leq b_1,\ldots,b_k \leq r - 1 \\ l_1,\ldots,l_k \geq 0}}\,\, \sum_{\substack{1 \leq j_1,\ldots,j_{h} \leq r - 1 \\ o_1,\ldots,o_{h} \geq 0}}   \\
& \quad \times \frac{s^{-(k + h)}}{k!h!}\,r^{\sum_{i = 1}^n \mu_i/r + \sum_{c = 1}^{k} (b_{c} + s)/r + \sum_{c = 1}^{h} (j_{c} + s)/r} \\
& \quad \times (\sigma p_d)^{\sum_{i = 1}^n \mu_i/r + \sum_{c = 1}^{k} (s - r + b_{c})/r + \sum_{c = 1}^{h} (s - r + j_{c})/r} q_s^{\sum_{i = 1}^n \mu_i/s + \sum_{c = 1}^{k} (b_{c} - r)/s + \sum_{c = 1}^{h} (j_{c} - r)/s} \\ 
& \quad \times \bigg(\int_{\overline{\mathcal{M}}_{g,n + k + h}} \frac{\Omega^{(r,s)}_{g;-\overline{\mu},r - \boldsymbol{b},r - \boldsymbol{j}}}{\prod_{i = 1}^n \left(1 - \frac{\mu_i}{r} \psi_i\right)}\bigg)  \prod_{c = 1}^{k} \frac{\mathcal{Q}_{s - r + b_c}^{(l_{c})}}{(l_{c} + 1)!}\,\prod_{c = 1}^{h} \frac{\mathcal{F}_{j_{c}}^{(o_{c})}}{(o_c + 1)!}. 
\end{split}
\end{equation}
The $\mathcal{Q}$-factors contain monomials in the $q$-variables (Lemma~\ref{QTev}), while the $\mathcal{F}$-factors contain monomials in the $p$- and $q$-variables (Lemma~\ref{Fll}). The former get replaced by a sum over $k$-tuple of non-empty $(s - 1)$-partitions $\rho^{(1)},\ldots,\rho^{(k)}$ such that $|\rho^{(c)}| = s - r + b_c \in \{1,\ldots,s-1\}$, that is $\boldsymbol{\rho} \in \mathcal{P}\!\!\mathcal{P}_{s - 1}^k$. The latter get replaced by a sum over a $h$-tuple of  $(s - 1)$-partitions $(\theta^{(1)},\ldots,\theta^{(h)})$ and a $h$-tuple of non-empty $(d - 1)$-partitions $(\gamma^{(1)},\ldots,\gamma^{(h)})$ such that $|\theta^{(c)}| + s |\gamma^{(c)}| = j_c + s$. The range of $j_c$ in \eqref{Hgn2nd} is enforced by the condition 
$$
|\theta^{(c)}| + s|\gamma^{(c)}|  \in \{s +1,\ldots,s + r - 1\}.
$$
\begin{definition} \label{thetrho} We call $\mathcal{P}_{s -1,d -1}^{+r}$ the set of ordered pairs $(\theta,\gamma)$ such that $\theta$ is a $(s - 1)$-partition, $\gamma$ is a non-empty $(d - 1)$-partition and $s + 1 \leq |\theta| + s|\gamma| \leq s + r - 1$.
\end{definition}

Then, the powers of $p_1,\ldots,p_{d - 1}$ and $q_{1},\ldots,q_{s - 1}$ recombine into
\begin{equation}
\label{qpmono} \vec{q}_{\lambda} \cdot \vec{p}_{\pi}
\end{equation}
with
\begin{equation}
\label{concatcun}
\sqcup \boldsymbol{\gamma} = d - \pi \qquad \text{and} \qquad \sqcup \boldsymbol{\rho} \sqcup \boldsymbol{\theta} = s - \lambda.
\end{equation}
We should not forget the natural combinatorial factor
$$
\frac{1}{k!h!},
$$
corresponding to permutations among the $k$-tuple $\boldsymbol{\rho}$ and the ordered pair of $h$-tuples $(\boldsymbol{\theta},\boldsymbol{\gamma})$. 

\medskip

We also have the factors $p_d^{\deg(p_d)}q_s^{\deg(q_s)}$ with (possibly negative) powers
\begin{equation}
\label{dpdq}
\begin{split} 
\deg(p_d) & =  \frac{(2g - 2 + n)s}{r} + \frac{|\mu|}{r} + \sum_{c = 1}^{k} \frac{|\rho^{(c)}|}{r} + \sum_{c = 1}^{h} \bigg( \frac{|\theta^{(c)}| + s |\gamma^{(c)}|}{r} - \ell(\gamma^{(c)})\bigg), \\  
\deg(q_s) & = \frac{|\mu|}{s} + \sum_{c = 1}^{k} \bigg(\frac{|\rho^{(c)}|}{s} - \ell(\rho^{(c)})\bigg) + \sum_{c = 1}^{h} \bigg( \frac{|\theta^{(c)}|}{s} - \ell(\theta^{(c)})\bigg).
\end{split}
\end{equation}
Due to \eqref{concatcun}, we have
\begin{equation*}
\begin{split}
|\pi| & = \sum_{c = 1}^{k} |d - \gamma^{(c)}|  = \sum_{c = 1}^{h} \big(d\ell(\gamma^{(c)}) - |\gamma^{(c)}|\big) \\
|\lambda| & = \sum_{c = 1}^{k} |s - \rho^{(c)}| + \sum_{c = 1}^{h} |s - \theta^{(c)}| = \sum_{c = 1}^{k} \big(s\ell(\rho^{(c)}) - |\rho^{(c)}|\big) + \sum_{c = 1}^{h} \big(s \ell(\theta^{(c)}) - |\theta^{(c)}|\big).
\end{split}
\end{equation*} 
Using as well $r = ds$, we can then rewrite \eqref{dpdq} as
\begin{equation}
\label{dpdqfinal}
d \cdot \deg(p_d) = 2g - 2 + n + \frac{|\mu| + |s - \lambda|}{s} - |\pi| \qquad \text{and}\qquad  \deg(q_s) = \frac{|\mu| - |\lambda|}{s}.
\end{equation}
We may also absorb the powers of $p_d$ and $q_s$ by considering the extended $s$-partition $\overline{\lambda}$ obtained from $\lambda$ by adding $\deg(q_s) \in \mathbb{Z}$ parts $s$, and the extended $d$-partition $\overline{\pi}$, obtained from $\pi$ by adding $\deg(p_d) \in \mathbb{Z}$ parts $d$. From \eqref{dpdqfinal}, it follows that their respective sizes are
$$
|\overline{\lambda}| = |\mu| \qquad \text{and} \qquad |\overline{\pi}| = 2g - 2 + n + \frac{|\mu| + |s - \lambda|}{s}.
$$

\medskip

Let us now turn to the other factors. First, we have an $\Omega$-integral of the form
$$
\int_{\overline{\mathcal{M}}_{g,n + k + h}} \frac{\Omega^{(r,s)}_{g;-\overline{\mu},s-|\boldsymbol{\rho}|,s(d + 1 - |\boldsymbol{\gamma}|) - |\boldsymbol{\theta}|}}{\prod_{i = 1}^n \left(1 - \frac{\mu_i}{r} \psi_i\right)},
$$
where we recall that $n = \ell(\mu)$. Second, from $\mathcal{F}$s and $\mathcal{Q}$s we have a minus sign to the power
$$
\deg(-1) = \sum_{c = 1}^{k} \big(\ell(\rho^{(c)}) - 1\big) + \sum_{c = 1}^{h} \big(\ell(\gamma^{(c)}) - 1\big) = \ell(\pi) - h + \sum_{c = 1}^{k} \big(\ell(\rho^{(c)}) - 1\big).
$$ 
Third, the $\mathcal{F}$s and the first and third line of \eqref{Hgn2nd} give rise to a factor  $\sigma^{\deg(\sigma)}$ with
\begin{equation*}
\begin{split}
\deg(\sigma) & = \frac{s}{r}(2g - 2 + n) + \frac{|\mu|}{r} + \sum_{c = 1}^{k} \frac{|\rho^{(c)}|}{r} + \sum_{c = 1}^{h} \frac{|\theta^{(c)}| + s|\gamma^{(c)}|}{r} \\
& = \frac{2g - 2 + n + |d - \pi|}{d} + \frac{|\mu| + |s - \lambda|}{r} = \frac{|d - \pi| + |\overline{\pi}|}{d},
\end{split}
\end{equation*}
using again \eqref{concatcun}. Third, we bring $s^{-h}$ in the second line of \eqref{Hgn2nd} as an extra $s^{-1}$ into each of the factor $\mathcal{F}_{j_c}^{(o_c)}$ for $c \in \{1,\ldots,h\}$, and more precisely into its factor $r - j_c$, in order to get a factor
\begin{equation}
\label{rdiffjc}
\prod_{c = 1}^{h} \frac{r - j_c}{s} = \prod_{c = 1}^{h} \bigg(d + 1 - |\gamma^{(c)}| - \frac{|\theta^{(c)}|}{s}\bigg).
\end{equation}
Fourth, the factorials and symmetry factors in $\mathcal{Q}$s and $\mathcal{F}$s yield a factor
$$
\prod_{c = 1}^{k} \frac{\big[\frac{s - |\rho^{(c)}|}{r}\big]_{\ell(\rho^{(c)}) - 1}}{|\text{Aut}(\rho^{(c)})|} \,\,\cdot   \prod_{c = 1}^{h} \frac{(d - |\gamma^{(c)}|)!}{(d + 1 - |\gamma^{(c)}| - \ell(\theta^{(c)}))!} \cdot \frac{\big[\frac{d + 1 - |\gamma^{(c)}|}{d}\big]_{\ell(\gamma^{(c)}) - 1}}{|\text{Aut}(\gamma^{(c)})| \cdot |\text{Aut}(\theta^{(c)})|}.
$$
Fifth,  taking into account the factor already absorbed in the previous step, we have a remaining factor of $s$ to the power
$$
\deg(s) = -(2g - 2 + n + k).
$$
Sixth, on top of the combinatorial factor
$$
\prod_{i = 1}^{n} \frac{(\mu_i/r)^{\lfloor \mu_i/r\rfloor}}{\lfloor \mu_i/r \rfloor !},
$$
we also get an extra factor of $r$ to the power
\begin{equation*}
\begin{split}
\deg(r) & = \bigg(1 + \frac{s}{r}\bigg)(2g - 2 + n) + \sum_{c = 1}^{k} \frac{|\rho^{(c)}| + r}{r} + \sum_{c = 1}^{h} \frac{|\theta^{(c)}| + s|\gamma^{(c)}|}{r} \\
& = \bigg(1 + \frac{s}{r}\bigg)(2g - 2 + n) + k + \frac{|\mu| + |s - \lambda|}{r} + \frac{|d - \pi|}{d} \\
& = -\deg(s) + \deg(\sigma).
\end{split}
\end{equation*}

\medskip

All together, this leads to the formula
\begin{equation}
\label{thevzn} 
\begin{split} 
H_{g,n}^{\text{2nd}}(\mu) & = (r/s)^{2g -2 + n} \prod_{i = 1}^{n} \frac{(\mu_i/r)^{\lfloor \mu_i/r \rfloor}}{\lfloor \mu_i/r \rfloor!} \sum_{\substack{\overline{\lambda} \in \overline{\mathcal{P}}_{s} \\ |\overline{\lambda}| = |\mu|}} \,\,\sum_{\substack{\overline{\pi} \in \overline{\mathcal{P}}_{d} \\ |\overline{\pi}| = 2g - 2 + n + (|\mu| + |s - \lambda|)/s}}  \!\!\!\!\!\!\!\!\!\!\!\!\!\! (\sigma r)^{(|d - \pi| + |\overline{\pi}|)/d} \cdot \vec{q}_{\overline{\lambda}} \cdot \vec{p}_{\overline{\pi}} \\ 
& \quad \times \Bigg(\sum_{h = 0}^{\ell(\pi)}\,\,\,  \sum_{k = 0}^{\ell(\lambda) - \ell(\pi)}  \frac{(-1)^{h - \ell(\pi)} (r/s)^{k}}{k!h!} \sum_{\substack{\boldsymbol{\rho} \in \mathcal{P}\!\!\mathcal{P}_{s - 1}^{k} \\ (\boldsymbol{\theta},\boldsymbol{\gamma}) \in (\mathcal{P}_{s - 1,d - 1}^{+r})^h \\ \sqcup \boldsymbol{\gamma} = d - \pi \\ \sqcup \boldsymbol{\theta} \sqcup \boldsymbol{\rho} = s - \lambda}} (-1)^{\sum_{c = 1}^{k} (\ell(\rho^{(c)}) - 1)}  \\
& \quad \times
\prod_{c = 1}^{k} \frac{\big[\frac{s - |\rho^{(c)}|}{r}\big]_{\ell(\rho^{(c)}) - 1}}{|\text{Aut}(\rho^{(c)})|} \, \cdot  \prod_{c = 1}^{h} \frac{(d - |\gamma^{(c)}|)!}{(d + 1 - |\gamma^{(c)}| - \ell(\theta^{(c)}))!} \cdot \frac{\big[\frac{d + 1 - |\gamma^{(c)}|}{d}\big]_{\ell(\gamma^{(c)}) - 1}}{|\text{Aut}(\gamma^{(c)})| \cdot |\text{Aut}(\theta^{(c)})|} \\
& \quad \times \prod_{c = 1}^{h}  \bigg(d + 1 - |\gamma^{(c)}| - \frac{|\theta^{(c)}|}{s}\bigg) \cdot \int_{\overline{\mathcal{M}}_{g,n + k + h}} \frac{\Omega^{(r,s)}_{g;-\overline{\mu},s - |\boldsymbol{\rho}|,s(d + 1 - |\boldsymbol{\gamma}|) - |\boldsymbol{\theta}|}}{\prod_{i = 1}^n \left(1 - \frac{\mu_i}{r} \psi_i \right)}\Bigg).
\end{split} 
\end{equation}
In this formula we use the convention that if $\pi$ is empty, \textit{i.e.} $\overline{\pi}$ only has parts $d$, then the sum over $h$ only contains $h = 0$; if $\pi$ is non-empty the sum should start at $h = 1$. Likewise, if $\lambda$ is empty, the sum over $k,h$ only contains the $k = h = 0$ term, which is equal to $1$.
 
\vspace{0.2cm}

\subsection{Structure of the formula}
 
The formula \eqref{thevzn} is heavy but we can stress a few structural properties. First, the factor in front of the integral in the last line of \eqref{thevzn} is the index of the corresponding insertion in the $\Omega$-class divided by $s$. These insertions take values in $\{1,\ldots,r - 1\}$ due to the constraints on the size of $|\theta^{(c)}|$ and $|\gamma^{(c)}|$ in Definition~\ref{thetrho}. Furthermore, the insertions $s - |\rho^{(c)}|$ take values in the smaller range $\{1,\ldots,d - 1\}$.

\medskip

We recall that the powers of $p_d$ and $q_s$ are given by \eqref{dpdqfinal}, namely
\begin{equation}
\begin{split}
\deg(p_d) & = \frac{2g - 2 + n - |\pi|}{d} + \frac{|\mu| + |s - \lambda|}{r},  \\
\deg(q_s) & = \frac{|\mu| - |\lambda|}{s},
\end{split}
\end{equation}
and they may be negative. The formula for $\deg(q_s)$ is as expected. The modular condition \eqref{eq:acondition} for the insertions in the $\Omega$-class says that
\begin{equation*}
\begin{split}
(2g - 2 + n + k + h)s & = -|\mu| + sk - \sum_{c = 1}^{k} |\rho^{(c)}| + s(d + 1)h - s\sum_{c = 1}^{h} \big(|\theta^{(c)}| + s|\gamma^{(c)}|\big) \quad \text{mod}\,\,r \\
& = -|\mu| + s(k + h) - |s - \lambda| - |d - \pi| \quad \text{mod}\,\,r.
\end{split}
\end{equation*}
Since $r = ds$, this is equivalent to $\deg(p_d)$ being an integer, so $\overline{\pi}$ is indeed an extended partition.

\medskip

In a few instances the sums simplify
\begin{itemize}
\item If $\max_{i \neq j} (\pi_i + \pi_j) \leq d$, the only surviving term in the $h$-sum is $h = \ell(\pi)$.
\item If $\max_{i \neq j} (\lambda_i + \lambda_j) \leq s$, the only surviving term is $k + h = \ell(\lambda)$.
\item If both conditions are satisfied, we must have $k = \ell(\lambda) - \ell(\pi)$ and $h = \ell(\pi)$, so the sum over $k,h$ consists of $k!h!$ terms which are all equal, therefore yields a single $\Omega$-integral.
\end{itemize}

\vspace{0.2cm}

\subsection{Polynomiality and vanishing}
\label{sec:polysecond}

The spectral curve $\mathscr{S}_{1}^{\text{2nd}}$ coincides with the family  $\mathscr{S}^{(\hat{\psi},\hat{y})}$ introduced in Section~\ref{sec:wHn} and \eqref{PQpsi}-\eqref{Shatpsi}, up to multiplying $P$ with $\sigma$. Therefore, $H_{g,n}^{\text{2nd}}(\mu)$ can be interpreted in terms of weighted Hurwitz numbers. Accordingly $H_{g,n}^{\text{2nd}}$ cannot contain terms with $\deg(p_d) < 0$ or $\deg(q_s) < 0$, and we get vanishing relations for the last three lines of \eqref{thevzn}. They can be reformulated in terms of the $(s - 1)$-partition $\nu = s - \lambda$ and the $(d - 1)$-partition $\tau = d - \pi$. Taking into account divisibility as in Section~\ref{sec:polyfirst}:
\begin{equation}
\label{degneg}
\begin{split}
\deg(p_d) < 0 & \quad \Longleftrightarrow \quad (2g -2 + n + |\tau|)s + |\mu| + |\nu| \leq d(\ell(\tau) - 1), \\
\deg(q_s) < 0  & \quad \Longleftrightarrow \quad |\mu| + |\nu| \leq s(\ell(\nu) - 1).
\end{split}
\end{equation}

\begin{theorem}
\label{genev} Let $d \geq 2$, $s \geq 1$ and set $r = ds$. Let $g \geq 0$. Let $\mu$ be a non-empty partition, $\nu$ is a (possibly empty) $(s - 1)$-partition, and $\tau$ a (possibly empty) $(d - 1)$-partition. Assuming that one of the two conditions \eqref{degneg} is satisfied, we have
\begin{equation*}
\begin{split}
0 =& \quad  \sum_{h = 0}^{\ell(\tau)}\,\, \sum_{k = 0}^{\ell(\nu) - \ell(\tau)} \frac{(-1)^{\ell(\tau) - h} (r/s)^k}{h!k!} \sum_{\substack{\boldsymbol{\rho} \in \mathcal{P}\!\!\mathcal{P}_{s - 1}^{k} \\ (\boldsymbol{\theta},\boldsymbol{\gamma}) \in (\mathcal{P}_{s - 1,d - 1}^{+r})^h \\ \sqcup \boldsymbol{\gamma} = \tau \\ \sqcup \boldsymbol{\theta} \sqcup \boldsymbol{\rho} = \nu}} (-1)^{\sum_{c = 1}^{k} (\ell(\rho^{(c)}) - 1)} \\
& \quad \times \prod_{c = 1}^{k} \frac{\big[\frac{s - |\rho^{(c)}|}{r}\big]_{\ell(\rho^{(c)}) - 1}}{|\text{Aut}(\rho^{(c)})|} \,\,\times   \prod_{c = 1}^{h} \frac{(d - |\gamma^{(c)}|)!}{(d + 1 - |\gamma^{(c)}| - \ell(\theta^{(c)}))!} \cdot \frac{\big[\frac{d + 1 - |\gamma^{(c)}|}{d}\big]_{\ell(\gamma^{(c)}) - 1}}{|\text{Aut}(\gamma^{(c)})| \cdot |\text{Aut}(\theta^{(c)})|} \\
& \quad \times \prod_{c = 1}^{h}  \bigg(d + 1 - |\gamma^{(c)}| - \frac{|\theta^{(c)}|}{s}\bigg) \cdot \int_{\overline{\mathcal{M}}_{g,n + k + h}} \frac{\Omega^{(r,s)}_{g;-\overline{\mu},s - |\boldsymbol{\rho}|,s(d + 1 - |\boldsymbol{\gamma}|) - |\boldsymbol{\theta}|}}{\prod_{i = 1}^n \left(1 - \frac{\mu_i}{r} \psi_i \right)}.
\end{split} 
\end{equation*}
\end{theorem} 
For given $\mu,g,n$, one can always find a vanishing relation by taking $\nu$ and $\tau$ large enough. 

\medskip

Let us spell this out in the simpler case $s = 1$, i.e. $r = d$. Then we do not have any $\boldsymbol{\rho}$ and $\boldsymbol{\theta}$, and $\deg(q_1)$ is always nonnegative. Taking to account the remark about $s = 1$ at the end of Section~\ref{S41}, the formula \eqref{thevzn} then simplifies:
\begin{equation*}
\begin{split}
H_{g,n}^{\text{2nd}}(\mu) & = r^{2g - 2 + n} \prod_{i = 1}^{n} \frac{(\mu_i/r)^{\lfloor \mu_i/r \rfloor}}{\lfloor \mu_i/r \rfloor !} \sum_{\substack{\overline{\pi} \in \overline{\mathcal{P}}_{r} \\ |\overline{\pi}| = 2g - 2 + n + |\mu|}} (\sigma r)^{\ell(\pi) + (2g - 2 + n + |\mu| - |\pi|)/r}  \cdot q_1^{|\mu|} \cdot \vec{p}_{\overline{\pi}} \\
& \quad \times \left(\sum_{h = 0}^{\ell(\pi)} \frac{(-1)^{\ell(\pi) - h}}{h!} \sum_{\substack{\boldsymbol{\gamma} \in \mathcal{P}_{r - 1}^{h} \\ 2 \leq |\gamma^{(c)}|  \leq r - 1 \\ \sqcup \boldsymbol{\gamma} = r - \pi}} \frac{\big[\frac{r + 1 - |\gamma^{(c)}|}{r}\big]_{\ell(\gamma^{(c)}) - 1}}{|\text{Aut}(\gamma^{(c)})|} \int_{\overline{\mathcal{M}}_{g,n + h}} \frac{\Omega^{(r,1)}_{g;-\overline{\mu},r + 1 - |\boldsymbol{\gamma}|}}{\prod_{i = 1}^n \left(1 - \frac{\mu_i}{r}\psi_i\right)}\right).
\end{split}
\end{equation*}
As usual, if $\pi$ is empty the sum over $h$ reduces to $h = 0$; if $\pi$ is non-empty it starts at $h = 1$. Under the condition $\deg(p_d) < 0$, that is
$$
2g - 2 + n + |\tau| + |\mu| \leq d(\ell(\tau) - 1).
$$ 
with $\tau = d - \pi$, we get
$$
\sum_{h = 0}^{\ell(\tau)} \frac{(-1)^{\ell(\tau) - h}}{h!} \sum_{\substack{\boldsymbol{\gamma} \in \mathcal{P}_{r - 1}^{h} \\ 2 \leq |\gamma^{(c)}| \leq r - 1 \\ \sqcup \boldsymbol{\gamma} = \tau}} \frac{\big[\frac{r + 1 - |\gamma^{(c)}|}{r}\big]_{\ell(\gamma^{(c)}) - 1}}{|\text{Aut}(\gamma^{(c)})|} \int_{\overline{\mathcal{M}}_{g,n+h}} \frac{\Omega^{(r,1)}_{g;-\overline{\mu},r + 1 - |\boldsymbol{\gamma}|}}{\prod_{i = 1}^n \left(1 - \frac{\mu_i}{r} \psi_i\right)}.
$$
Here, the requirement that $\gamma^{(c)}$ has size $\{2,\ldots,r - 1\}$ is implied by Definition~\ref{thetrho} for empty $\theta^{(c)}$. So, if $\pi$ has parts of size $1$ they cannot appear alone in a $\gamma^{(c)}$. In particular, if $\pi$ has no parts $1$ and $\tau$ is bounded, i.e. $\min_{i \neq j} (\tau_i + \tau_j) \geq r$, there only remains a single term
$$
\int_{\overline{\mathcal{M}}_{g,n + \ell(\tau)}} \frac{\Omega^{(r,1)}_{g;-\overline{\mu},r + 1 - \tau}}{\prod_{i = 1}^{n} \left(1 - \frac{\mu_i}{r} \psi_i\right)} = 0.
$$  
This proves Theorem~\ref{th:B}.

\vspace{1cm}
\section{The third deformation}
\label{sec:Deformation:three}
\vspace{0.5cm}

\subsection{Setting and vanishing result} Let $r \geq 2$ and $s \in \{1,\ldots,r - 1\}$. Contrarily to Sections~\ref{sec:Deformation:one}-\ref{sec:Deformation:two}, we do not assume any divisibility condition between $r$ and $s$. We consider the spectral curve
\begin{equation}
\label{St3rd}
\mathscr{S}_{t}^{\text{3rd}} \,: \quad \left\{\begin{array}{lll} x_t(z) & = & \ln z - \big(p_r z^r + t \tilde{P}(z)\big) \\[2pt] y_t(z) & = & z^{s} \end{array}\right.
\end{equation}
with 
$$
\tilde{P}(z) = \sum_{j = 1}^{r - s} p_j z^j.
$$
Let $F_{g,n}^{t}(z_1,\ldots,z_n)$ be the free energies associated to this curve. The computations with this deformation follow the same steps as in Sections~\ref{sec:Deformation:one}-\ref{sec:Deformation:two} without additional difficulty --- here it is important that  $\tilde{P}$ has degree at most $r - s$ instead of $r - 1$, otherwise the simplifying tricks used in the proofs of Lemma~\ref{QTev} (see also \cite[Lemma 5.3]{BDKLM}) or Lemma~\ref{Fll} would not work. We will therefore omit the details and only state the results.

\medskip

The deformation $1$-form is
$$
\eta_t(z) = -\Res_{w = \infty} \omega_{0,2}(z,w) f_t(w),\qquad \dd f_t(z) = s \tilde{P}(z) z^{s - 1}\dd z,
$$
The evaluation of the Taylor series at $t = 1$ yields an expansion as $z_i \rightarrow 0$
$$
F_{g,n}^{1}(z_1',\ldots,z_n') \approx \sum_{\mu_1,\ldots,\mu_n > 0} H_{g,n}^{\text{3rd}}(\mu_1,\ldots,\mu_n) \prod_{i = 1}^n e^{\mu_i x_1(z_i')},
$$
with
\begin{equation}
\label{Hgn3rddd} 
\begin{split}  
H_{g,n}^{\text{3rd}}(\mu) & =  (s^{-1}r^{1 + s/r})^{2g - 2 + n}  \prod_{i = 1}^n \frac{(\mu_i/r)^{\lfloor \mu_i/r \rfloor}}{\lfloor \mu_i/r \rfloor!}  \!\!\!\!\!\sum_{\substack{\overline{\pi} \in \overline{\mathcal{P}}_{r} \\ \text{with parts in} \,\,\{1,\ldots,r - s\} \cup \{r\} \\ |\overline{\pi}| = (2g - 2 + n)s + |\mu|}} r^{(|\mu| + |r + s - \pi|)/r} \cdot \vec{p}_{\overline{\pi}} \\
& \quad \times \left(\sum_{k = 0}^{\ell(\pi)} \frac{1}{k!}  \sum_{\substack{\boldsymbol{\rho} \in \mathcal{P}_{s,r + s - 1}^{k} \\ |\rho^{(c)}| \leq r +s - 1 \\ \sqcup \boldsymbol{\rho} = r + s - \pi}} \prod_{c = 1}^{k} \frac{\big[\frac{r + s - |\rho^{(c)}|}{r}\big]_{\ell(\rho^{(c)}) - 1}}{|\text{Aut}(\rho^{(c)})|} \cdot \int_{\overline{\mathcal{M}}_{g,n + k}} \frac{\Omega^{(r,s)}_{g;-\overline{\mu},r + s - |\boldsymbol{\rho}|}}{\prod_{i = 1}^{n} \left(1 - \frac{\mu_i}{r} \psi_i\right)}\right),
\end{split}  
\end{equation}
where $\mathcal{P}_{s,r + s - 1}$ was defined in Section~\ref{Sec:not} as the set of partitions which are either empty or have parts of size $\{s,s + 1,\ldots,r + s - 1\}$. The constraints on the (possibly empty) $(r + s - 1)$-partitions $\rho^{(c)}$ imply that the corresponding insertions in the $\Omega$-class are in $\{1,\ldots,r\}$, but an index $r$ in the $\Omega$-class is equivalent to an index $0$, all insertions effectively remain in the fundamental range $\{0,\ldots,r-1\}$. For $|\rho^{(c)}| = r$ we have an insertion of an index $s$, which is the unit --- we recall from Theorem~\ref{CohFTl} that in the range $s \in \{1,\ldots,r-1\}$ considered here, $\Omega^{(r,s)}$ forms a cohomological field theory with flat unit.

\medskip

As $s$ may not divide $r$, the spectral curve \eqref{St3rd} is not of the form considered in \cite{bych-dun-kaz-sha} and we do not have \textit{a priori} a combinatorial meaning for $H_{g,n}^{\text{3rd}}(\mu)$ from which one could infer that it cannot contain negative powers of $p_1,\ldots,p_{r - s}$. Nevertheless, we justify in the next section that this is indeed the case, by a direct study of the topological recursion formula.

\begin{proposition}
\label{polynom3} For any $\mu_1,\ldots,\mu_n > 0$, $H_{g,n}^{\text{3rd}}(\mu)$ is a power series in the variables $p_1,\ldots,p_{r - s},p_r$.
\end{proposition}

By comparison with \eqref{Hgn3rddd} which may contain negative powers of $p_r$, we obtain vanishing relations. More precisely, the power of $p_r$ is
$$
\deg(p_r) = (2g -2 + n)s + |\mu| - |\pi| = (2g -2 + n) s + |\mu| - |r + s - \tau|
$$
in terms of the $(r + s - 1)$-partition $\tau = r  + s - \pi$. Due to the modular condition \eqref{eq:acondition}, we have $\deg(p_r) < 0$ is equivalent to
$$
(2g - 2 + n) + |\mu| + |\tau| \leq (r + s)\ell(\tau) - r
$$ 
with $n = \ell(\mu)$.  If this condition is satisfied, we then have the vanishing relation
$$ 
\sum_{k = 0}^{\ell(\tau)} \frac{1}{k!} \sum_{\substack{\boldsymbol{\rho} \in \mathcal{P}_{s,r + s - 1}^{k} \\ |\rho^{(c)}| \leq r + s - 1 \\ \sqcup \boldsymbol{\rho} = \tau}} \prod_{c = 1}^k \frac{\big[\frac{r + s - |\rho^{(c)}|}{r}\big]_{\ell(\rho^{(c)}) - 1}}{|\text{Aut}(\rho^{(c)})|} \int_{\overline{\mathcal{M}}_{g,\ell(\mu) + \ell(\tau)}} \frac{\Omega^{(r,s)}_{g;-\overline{\mu},r + s - |\boldsymbol{\rho}|}}{\prod_{i = 1}^{\ell(\mu)} \left(1 - \frac{\mu_i}{r} \psi_i\right)} = 0.
$$ 
Note that there is no alternating sign in this sum, as in Section~\ref{sec:Deformation:two} for the sum over $k$ when we deformed $P$, and unlike Section~\ref{sec:Deformation:one} when we deformed $Q$. The sum is non-empty only if the parts of $\tau$ all belong to $\{s,\ldots,r + s - 1\}$.  
Furthermore, if $\tau$ is bounded, \textit{i.e.} $\min_{i \neq j} (\tau_i + \tau_j) \geq r + s$, only $k = \ell(\tau)$ survives in the sum and we get the vanishing of a single $\Omega$-integral
$$
\int_{\overline{\mathcal{M}}_{g,n + \ell(\tau)}} \frac{\Omega^{(r,s)}_{g;-\overline{\mu},r + s - \tau}}{\prod_{i = 1}^{n} \left(1 - \frac{\mu_i}{r} \psi_i\right)} = 0.
$$  
This proves Theorem~\ref{th:C}.

\vspace{0.2cm}

\subsection{Proof of polynomiality (Proposition~\ref{polynom3})}

\label{sec:polynom3}

We keep the assumption $r \geq 2$ and $s \in \{1,\ldots,r - 1\}$. For the sake of clarity of our method, most of the argument will be given for the more general spectral curve
\begin{equation}
\label{spnung}\left\{\begin{array}{lll} x_t(z) & = &  \ln z - P(z) \\[2pt] y_t(z) & = &  z^s \end{array}\right. \qquad \text{with} \qquad P(z) = \sum_{j = 1}^{r} p_j z^j
\end{equation}
We want to examine the behavior of the correlators or free energies of topological recursion as functions of $p_r \rightarrow 0$. The ingredients of the topological recursion were listed in Section~\ref{TRreview}: we need to examine the behavior of ramification points (zeros of $x'(z)$, collected in the set $\mathcal{R}$), of the local involution $z \mapsto \overline{z}$ near ramification points, of the recursion kernel $K_{\alpha}(z_0,z)$. Additionally, as we are interested in the coefficients of decomposition of the free energies on a suitable basis of functions (or equivalently, its expansion as $z_i \rightarrow 0$), we must examine as well the behavior of this basis of functions as $p_r \rightarrow 0$.

\medskip

Note that for $s = 1$ (and for all cases such that $s|r$), this spectral curve is of the form studied in Section~\ref{sec:wHn} and Proposition~\ref{polynom3} is a consequence of \cite{bych-dun-kaz-sha}. Therefore, we can assume $s \geq 2$ and $r \geq 3$.

\subsubsection{The ramification points}

We look for the zeros of
$$
zx'(z) = 1 - zP'(z) = 1 - \sum_{j = 1}^{r} jp_jz^j,
$$
and want to understand their behavior as $p_r \rightarrow 0$. These are also the zeros of 
\begin{equation}
\label{inversez} z^{-r} - \sum_{j = 1}^{r} jp_j (z^{-1})^{r - j} = 0.
\end{equation}
The left-hand side is a polynomial of fixed degree $r$ in the variable $z^{-1}$, so its set of $r$ zeros is a continuous functions of $p_r$ in a neighborhood of $0$. In particular, at $p_r = 0$ it has a zero of order $r'$ at $z^{-1} = 0$, where
$$
r' = \min\big\{j \in \{1,\ldots,r - 1\} \,\,|\,\,p_{r - j} \neq 0\big\}.
$$
If the minimum does not exist we set $r' = r$.

\begin{definition}
We call \emph{escaping root} a zero of $1 - zP'(z)$ going to infinity as $p_r \rightarrow 0$. The other zeros are called \emph{regular roots} and they have a limit as $p_r \rightarrow 0$.
\end{definition}

\begin{lemma}\label{prop:regularroots.1}
Assume that $r' \in \{1,\ldots,r-1\}$ and $(1 - zP'(z))|_{p_r = 0}$ has only simple zeros distinct from $0$ (this is a generic condition). In the regime $p_r \rightarrow 0$, the escaping roots have a Puiseux series expansion of the form $\alpha_j = \sum_{k \geq -1} \alpha_{j,k} p_r^{k/r'}$ while the regular roots have a power series expansion.
\end{lemma}
\begin{proof}
Let $\beta_{r' + 1},\ldots,\beta_r$ be the roots of $(1 - zP'(z))\big|_{p_r = 0}$ and
$$
\forall j \in \{1,\ldots,r'\}\qquad \zeta_j = e^{2{\rm i}\pi j/r'} \bigg(-\frac{(r-r')p_{r - r'}}{r}\bigg)^{1/r'},
$$
for some fixed choice of $r'$-th root in the bracket. We can label $\alpha_{r' + 1},\ldots,\alpha_r$ the regular roots tending to $\beta_{r' +1},\ldots,\beta_r$, and $\alpha_{1},\ldots,\alpha_{r'}$ the escaping roots which are such that $\alpha_j \sim \zeta_j p_r^{-1/r'}$ --- an inspection of \eqref{inversez} reveals that they behave in this way. The assumptions we have taken imply that $\alpha_1,\ldots,\alpha_r$ are continuous functions of $p_r$ in a sectorial neighborhood of $0$. Notice that the zeros of any polynomial are algebraic functions of the coefficients of such polynomial, hence in particular algebraic functions of $p_r$. In the present case, where we can label the roots individually, it implies that $\alpha_1,\ldots,\alpha_r$ have Puiseux series expansions as $p_r \rightarrow 0$.

\medskip

Using the splitting between regular and escaping roots, Vi\`ete's formula yields for $i \in \{1,\ldots,r'\}$ 
$$
e_{i}(\alpha_1,\ldots,\alpha_{r'}) + e_i(\alpha_{r' + 1},\ldots,\alpha_{r}) + \sum_{\substack{i_1 + i_2 = i \\ i_1,i_2 \geq 1}} e_{i_1}(\alpha_{1},\ldots,\alpha_{r'})e_{i_2}(\alpha_{r' + 1},\ldots,\alpha_r) = \delta_{i,r'} (-1)^{r'} \frac{(r - r')p_{r - r'}}{rp_r}.
$$
We see this as a triangular system determining $e_i(\alpha_1,\ldots,\alpha_{r'})$ for $i \in \{1,\ldots,r'\}$, in the form
\begin{equation}
\label{eiaaa}
e_i(\alpha_1,\ldots,\alpha_{r'}) = \delta_{i,r'}(-1)^{r'} \frac{(r - r')p_{r - r'}}{rp_r}  + \sum_{\lambda \vdash i} C_{\lambda} e_{\lambda}(\alpha_{r' + 1},\ldots,\alpha_r)
\end{equation}
for some universal constants $C_{\lambda}$. Here, $e_{\lambda}$ are the elementary symmetric polynomials associated with a partition $\lambda$, and for $\lambda \vdash i$ they form a basis of the space of symmetric polynomials of degree equal to $i$.

\medskip

Next, we write the remaining Vi\`ete's formulae for $i \in \{r' + 1,\ldots,r\}$, in the form
$$
e_i(\alpha_{r' + 1},\ldots,\alpha_{r})  + \sum_{i' = 1}^{r'} e_{i'}(\alpha_1,\ldots,\alpha_{r'}) e_{i - i'}(\alpha_{r' + 1},\ldots,\alpha_r) = (-1)^{i} \frac{(r - i)p_{r - i}}{rp_r}.
$$
Then, we substitute \eqref{eiaaa} for $e_{i'}(\alpha_1,\ldots,\alpha_{r'})$, write $\alpha_j = \beta_j + \varkappa_j$ for $j \in \{r' + 1,\ldots,r\}$, plug an (unknown) formal power series expansion $\varkappa_j \approx \sum_{k \geq 1} \varkappa_{j,k} p_r^{k}$ in the equation, and extract the coefficient of $p_r^k$ for each $k \in \mathbb{Z}_{\geq 0}$. The result takes the form
$$
\forall i \in \{r'+1,\ldots,r\}\qquad \sum_{j = r' + 1}^{r} e_{i - 1}(\beta_{r' + 1},\ldots,\widehat{\beta_j},\ldots,\beta_{r}) \varkappa_{j,k} = \text{Polynomial}\big((\varkappa_{j,m})_{\substack{r' + 1 \leq j \leq r \\ 0 \leq m \leq k - 1}}\big).
$$
This is a invertible triangular system, since the matrix
$$
\big(e_{i - 1}(\beta_{r' + 1},\ldots,\widehat{\beta_j},\ldots,\beta_{r})\big)_{r' + 1 \leq i,j \leq r}
$$
has determinant proportional to the Vandermonde $\prod_{r' + 1 \leq j < l \leq r} (\beta_l - \beta_j)$ which is non-zero due to the assumption of distinct roots. By uniqueness of Puiseux series expansions, its solution determines the (formal power) series expansion for the regular roots.

\medskip

We now return to the Vi\`ete's formulae \eqref{eiaaa} indexed by $i \in \{1,\ldots,r'\}$ and insert there the power series expansion of the regular roots, the decomposition
$$
\forall j \in \{1,\ldots,r'\}\qquad \alpha_j = p_r^{-1/r'}\zeta_j + \varkappa_j,
$$
and a (unknown) regular formal Puiseux series expansion $\varkappa_j \approx \sum_{k \geq 0} \varkappa_{j,k} p_r^{k/r'}$. The leading term in the equation is of order $p_r^{-i/r'}$ and the equation is automatically satisfied at this order due to the properties of the $r'$-th roots. Extracting the coefficient of $p_r^{(k - i + 1)/r'}$ for some $k \in \mathbb{Z}_{\geq 0}$ yields a system of the form
$$
\forall i \in \{1,\ldots,r'\}\qquad \sum_{j = 1}^{r'} e_{i - 1}(\zeta_1,\ldots,\widehat{\zeta_j},\ldots,\zeta_{r'}) \varkappa_{j,k} = \text{Polynomial}\left((\varkappa_{j,m})_{\substack{1 \leq j \leq r \\ 0 \leq m \leq k - 1}}\,,\,(\varkappa_{j,m})_{\substack{r' + 1 \leq j \leq r \\ m \geq 0}}\right) .
$$
Since the $\zeta$s are pairwise distinct, this is again an invertible triangular system, whose solution determines uniquely the Puiseux series expansion of the escaping roots.
\end{proof}

 \subsubsection{The local involutions}
 \label{loclin}
 
 Recall that $\mathcal{R}$ is the set of zeros of $1 - zP'(z)$, and let $\alpha \in \mathcal{R}$. Compared to Section~\ref{TRreview}, it is convenient to recenter the local involution and define it as the holomorphic map  $z \mapsto \overline{z}$ locally defined near $0$ such that $x(\alpha + z) = x(\alpha + \overline{z})$ but $z \neq \overline{z}$ for $z \neq 0$. Equivalently, it is characterised by the condition $\dd z \neq {\dd \overline{z}}$ and
$$
z\cdot \frac{\prod_{\gamma \in \mathcal R \setminus \{\alpha\}} (z + \alpha - \gamma)}{z+\alpha}\,\dd z = \bar z \cdot \frac{\prod_{\gamma \in \mathcal R \setminus \{\alpha\}}(\bar z + \alpha - \gamma)}{\bar z+\alpha}\,\dd\bar z,
$$
We rewrite it as
\begin{equation}\label{eq:inv.1}
\frac {\bar z}{z}\cdot \frac{ \prod_{{\gamma \in \mathcal R \setminus \{\alpha\}}} (\bar z + \alpha - \gamma)}{\prod_{{\gamma \in \mathcal R \setminus \{\alpha\}}} (\bar z + \alpha - \gamma)} \cdot \frac{z + \alpha}{\bar z + \alpha } \cdot \frac{\dd \bar z}{\dd z } = 1.
\end{equation}
As a holomorphic function, this involution admits a series expansion
\begin{equation}
\label{barzoverz} \frac{{\bar z}}{z} \approx \sum_{k \geq 0}^\infty a_k z^k
\end{equation}
as $z \rightarrow 0$.  Studying the leading order of \eqref{eq:inv.1}, we get $a_0 = -1$.  We want to describe the structure of $(a_k)_{k \geq 1}$ as a function of the points in $\mathcal{R}$.

\begin{lemma}
\label{lem:involution} Keep the assumptions of Lemma~\ref{prop:regularroots.1} and assume as well that $\beta_{r' + 1},\ldots,\beta_{r}$ are non-zero (this is a generic condition). Then, for each $k \in \mathbb{Z}_{> 0}$, the coefficient $a_k$ has a regular Puiseux series expansion in the variable $p_r$.
\end{lemma}
\begin{proof} As $z \rightarrow 0$, we derive from \eqref{barzoverz} the expansion
$$
\frac{\dd\bar z}{\dd z} \approx \sum_{k \geq 0} (k+1) a_{k} z^k.
$$
and for any $b \in \mathbb{C}^*$
$$
 \frac{\bar z - b}{z-b} \approx \frac {-1}{b}\bigg(\sum_{k \geq 1} a_{k-1} z^{k} - b\bigg)\bigg(\sum_{m \geq 0} \frac{z^m}{b^m}\bigg) =  \sum_{k \geq 0} z^k\sum_{l = 0}^k \frac {-a_{l-1}(b)}{b^{k - l+1}} = 1 + O(z),
$$
where $a_l(b) = a_l$ for $l \geq 0$ and $a_{-1}(b) = -b$. Likewise:
$$
\frac{z - b}{\bar z - b} \approx \sum_{k \geq 0} z^k \sum_{\substack{m_1,m_2,\ldots \geq 0 \\ \sum_{i \geq 1} im_i = k}} \frac{\big(\sum_{i \geq 1} m_i\big)!}{\prod_{i \geq 1} m_i!} \cdot \prod_{i \geq 1} \bigg( \sum_{l = 0}^i \frac {a_{l-1}(b)}{b^{i - l + 1}}\bigg)^{m_i}.
$$
Equating the coefficients of the $z \rightarrow 0$ expansion on both sides of \eqref{eq:inv.1}, we infer from the listed expansions  that the $a_k$ for $k \geq 1$ are polynomials in $\alpha^{-1}$ and $(\alpha - \gamma)^{-1}$ for $\gamma \in \mathcal{R} \setminus \{\alpha\}$. In view of Lemma~\ref{prop:regularroots.1} and since we additionally assume that the regular roots have a non-zero limit, each of this quantity has a limit as $p_r \rightarrow 0$, and a Puiseux series expansion as $p_r \rightarrow 0$. Therefore, $a_k$ has a regular Puiseux series expansion as $p_r \rightarrow 0$.
\end{proof}

\subsubsection{The recursion kernel}
 
The first part of the topological recursion formula \eqref{Trformula} involves the recursion kernel at $\alpha \in \mathcal{R}$. Specialised to the spectral curve \eqref{spnung} and with the modified convention of Section~\ref{loclin} for the local involution, this is
\begin{equation*}
\begin{split}
K_{\alpha}(z_0,\alpha + z)\, \dd z\,\dd\overline{z} & = \frac{\dd z_0}{2}\bigg(\frac{1}{z_0 - (\alpha + z)} - \frac{1}{z_0 - (\alpha + \overline{z})}\bigg)\,\frac{\dd z\, \dd \overline{z}}{\big(y(\alpha + z) - y(\alpha + \overline{z})\big)\dd x(z)} \\ 
& = -\frac{1}{2}\,\frac{(z + \alpha)\,\dd z_0\,\dd \overline{z}}{rp_r z \prod_{\gamma \in \mathcal{R} \setminus \{\alpha\}} (z + \alpha - \gamma)} \,\frac{1}{(\alpha + z)^{s} - (\alpha + \overline{z})^s}\,\frac{z - \overline{z}}{(z_0 - z - \alpha)(z_0 - \overline{z} - \alpha)}.
\end{split} 
\end{equation*} 
Residues will be taken at $z =0$, and we included a factor $\dd z \dd \overline{z}$ as it will come up in the second part of \eqref{Trformula}.

\begin{lemma}
\label{lem:kernel}
For each $\alpha \in \mathcal{R}$, we have as $z \rightarrow 0$
$$
K_{\alpha}(z_0,\alpha + z)\, \dd z\,\dd \overline{z}\,\, \approx \sum_{k \geq -1} K_{\alpha,k}(z_0)\,z^{k}\, \dd z \,\dd z_0.
$$
The functions $K_{\alpha,m}(z_0)$ are polynomials in $(z_0 - \alpha)^{-1}$. Under the assumptions of Lemma~\ref{lem:involution} and assuming $s \in \{2,\ldots,r-1\}$, the coefficients of this polynomial have a regular Puiseux series expansion as $p_r \rightarrow 0$.
\end{lemma}
\begin{proof}
We rewrite $K_{\alpha}(z_0,\alpha + z)\,\dd z\,\dd \overline{z} = c_{\alpha}  z^{-1} (z_0 - \alpha)^{-2}\,\tilde{K}_{\alpha}(z_0,z) \dd z_0 \dd z$, where
\begin{equation}
\label{Ktildetal}
\begin{split} 
c_{\alpha} & := \frac{\alpha^{2 - s}}{2srp_r \prod_{\gamma \in \mathcal{R}\setminus \{\alpha\}} (\alpha - \gamma)}, \\
\tilde{K}_{\alpha}(z_0,z) & =  \frac{1 + z/\alpha}{\prod_{\gamma \in \mathcal{R} \setminus \{\alpha\}} \left(1 + \frac{z}{\alpha - \gamma}\right)}\cdot \frac{2sz}{\alpha \big((1 + z/\alpha)^{s} - (1 + \overline{z}/\alpha)^s\big)}\cdot \frac{1 - \overline{z}/z}{\left(1 - \frac{z}{z_0 - \alpha}\right)\left(1 - \frac{\overline{z}}{z_0 - \alpha}\right)} \cdot \frac{-\dd \overline{z}}{\dd z}.
\end{split}
\end{equation}

If $\alpha$ is a regular root, we have as $p_r \rightarrow 0$
$$
\frac{1}{\prod_{\gamma \in \mathcal{R} \setminus \{\alpha\}} (\alpha - \gamma)} = O(p_r)
$$
because there are $r'$ escaping roots behaving like $O(p_r^{-1/r'})$ contributing to the product. Under the assumptions of Lemma~\ref{lem:involution}, the numerator $\alpha^{2 - s}$ has a non-zero limit and thus $c_{\alpha} = O(1)$. If $\alpha$ is an escaping root, we rather have 
$$
\frac{1}{\prod_{\gamma \in \mathcal{R} \setminus \{\alpha\}} (\alpha - \gamma)}  = O(p_r^{(r - 1)/r'}).
$$
Then, $\alpha^{2 - s} = O(p_r^{-(2-s)/r'})$ and we obtain $c_{\alpha} = O(p_r^{(-r' + s - 2 + r - 1)/r'})$. The maximal value of $r'$ being $r - 1$, this gives $c_{\alpha} = O(p_r^{(s -2)/r'})$. Hence, for $s \geq 2$, we always have $c_{\alpha} = O(1)$.

\medskip

The function $\tilde{K}_{\alpha}(z_0,z)$ has a power series expansion in the variable $z \rightarrow 0$, and we claim that its coefficients are polynomial in $\alpha^{-1}$, $(\alpha - \gamma)^{-1}$ for $\gamma \in \mathcal{R} \setminus \{\alpha\}$, and $(z_0 - \alpha)^{-1}$. Indeed, the proof of Lemma~\ref{lem:involution} showed that the coefficients $(a_k)_{k \geq 1}$ of expansion of $\overline{z}$ in the $z$-variable have this type of property. We also compute as $z \rightarrow 0$
$$
\alpha \big((1 + z/\alpha)^{s} - (1 + \overline{z}/\alpha)^{s}\big) = \sum_{l = 1}^{s}  {s \choose l} \alpha^{1 - l} (z^l - \overline{z}^l) \approx 2sz + \sum_{k \geq 1} A_k z^{k + 1},
$$
where $A_k$ is a polynomial in the variables $(a_l)_{l \geq 1}$. Therefore, as $z \rightarrow 0$:
\begin{equation}
\label{1zs}
\frac{2sz}{\alpha\big((1 +  z/\alpha)^{s} - (1 + \overline{z}/\alpha)^{s}\big)} \approx 1 + \sum_{k \geq 1} z^k \sum_{\substack{m,k_1,\ldots,k_m \geq 1 \\ k_1 + \cdots + k_m = k}} \frac{(-1)^m}{(2s)^m} \prod_{i = 1}^{m} A_{k_i},
\end{equation}
and the coefficient of $z^k$ there is a polynomial in $(a_k)_{k \geq 1}$ and in $\alpha^{-1}$. The other factors in \eqref{Ktildetal} are easier to analyse. This shows the claim. In the proof of Lemma~\ref{lem:involution} we showed that $\alpha^{-1}$ and $(\alpha - \gamma)^{-1}$ have a regular Puiseux series expansion in the variable $p_r$. Therefore, the coefficient of $z^k$ in $\tilde{K}_{\alpha}(z_0,z)$ is a polynomial in $(z_0 - \alpha)^{-1}$ whose coefficients have regular Puiseux series expansions as $p_r \rightarrow 0$. Since we already proved that $c_{\alpha}$ has this property, this proves the lemma.
\end{proof}

\subsubsection{The free energies}

For $2g - 2 + n > 0$, the free energies $F_{g,n}(z_1,\ldots,z_n)$ are rational functions on the Riemann sphere with poles at $z_i = \alpha$ for some $\alpha \in \mathcal{R}$ and $i \in \{1,\ldots,n\}$, and zeros at $z_i = 0$. We can therefore decompose them as
\begin{equation}
\label{freener}
F_{g,n}(z_1,\ldots,z_n) = \sum_{\substack{1 \leq j_1,\ldots,j_n \leq r \\ k_1,\ldots,k_n \geq 0}} G_{g,n}\big[\begin{smallmatrix} j_1 & \cdots & j_n \\ k_1 & \cdots & k_n \end{smallmatrix}\big] \prod_{i = 1}^n \bigg(\frac{1}{(z_i - \alpha_{j_i})^{k_i + 1}} - \frac{1}{(-\alpha_{j_i})^{k_i + 1}}\bigg),
\end{equation}
where only finitely many coefficients are non-zero. In general, if $u$ is an independent variable, let us write as $z \rightarrow 0$
$$
\frac{1}{z - u} - \frac{1}{-u} \approx  \sum_{m \geq 1} E_m(u) e^{mx(z)}.
$$ 
We compute for any $m \geq 1$:
\begin{equation*}
\begin{split}
E_m(u) & = \Res_{z = 0} e^{-(m + 1)x(z)}\big(\frac{1}{z - u} - \frac{1}{-u}\bigg) \dd(e^{x(z)}) = -\frac{1}{m} \Res_{z = 0} \frac{e^{-mx(z)}\dd z}{(z - u)^2} \\
& = - \frac{1}{m} \Res_{z = 0} \frac{z^{-m}e^{-m \sum_{j = 1}^{r} p_j z^j}}{(z - u)^2} = - \frac{1}{m} \big[z^{-(m - 1)}\big]\,\frac{e^{- m\sum_{j = 1}^{r} p_j z^j}}{(z - u)^2}.
\end{split}
\end{equation*}
Without making it more explicit, we can already tell that $E_m(u)$ is a polynomial in the variables $1/u$ and $p_1,\ldots,p_r$. Then, by differentiating $k$ times with respect to $u$, we get
$$
\frac{1}{(z - u)^{k + 1}} - \frac{1}{(-u)^{k + 1}} \approx \sum_{m \geq 1} \frac{E_m^{(k)}(u)}{k!}\,e^{mx(z)},
$$
and $\partial_u^k E_m(u)$ is also a polynomial in the variables $1/u$ and $p_1,\ldots,p_r$. This allows us extracting the expansion of the free energies as $z_i \rightarrow 0$ in the variables $e^{x(z_i)}$ :
$$
F_{g,n}(z_1,\ldots,z_n) \approx \sum_{\mu_1,\ldots,\mu_n > 0} H_{g,n}(\mu_1,\ldots,\mu_n) \prod_{i = 1}^n e^{\mu_i x(z_i)},
$$
where
$$
H_{g,n}(\mu_1,\ldots,\mu_n) = \sum_{\substack{1 \leq j_1,\ldots,j_n \leq r \\ k_1,\ldots,k_n \geq 0}} G_{g,n}\big[\begin{smallmatrix} j_1 & \cdots & j_n \\ k_1 & \cdots & k_n \end{smallmatrix}\big] \prod_{i = 1}^{n} \frac{E_{\mu_i}^{(k_i)}(\alpha_{j_i})}{k_i!}
$$
and we recall there are only finitely many non-zero terms in this sum. Since $\alpha^{-1}$ has a regular Puiseux series expansion as $p_r \rightarrow 0$ for any $\alpha \in \mathcal{R}$, it is sufficient to check that the coefficients $C_{g,n}\big[\begin{smallmatrix} j_1 & \cdots & j_n \\ k_1 & \cdots & k_n \end{smallmatrix}\big]$ have a regular Puiseux expansion as $p_r \rightarrow 0$ to conclude the same for $H_{g,n}(\mu_1,\ldots,\mu_n)$. 

\medskip

We first compute the cases $2g -2  +n = 1$. For $(g,n) = (0,3)$ we have
\begin{equation*}
\begin{split}
\omega_{0,3}(z_1,z_2,z_3) & = \sum_{\alpha \in \mathcal{R}} \Res_{z = 0} K_{\alpha}(z_1,\alpha + z) \bigg(\frac{\dd z\,\dd \overline{z}\,\dd z_2\,\dd z_3}{(\alpha + z - z_2)^2(\alpha + \overline{z} - z_3)^2}  + (z_2 \leftrightarrow z_3)\bigg)   \\ 
& = \sum_{\alpha \in \mathcal{R}} \frac{2c_{\alpha}}{(z_1 - \alpha)^2(z_2 - \alpha)^2(z_3 - \alpha)^3} \\
& = \dd_{z_1}\dd_{z_2}\dd_{z_3}\bigg(\sum_{\alpha \in \mathcal{R}}  - 2c_{\alpha}\prod_{i = 1}^3 \big((z_1 - \alpha)^{-1} - (-\alpha^{-1})\big)\bigg).
\end{split}
\end{equation*}
In other words $G_{0,3}\big[\begin{smallmatrix} j_1 & j_2 & j_3 \\ k_1 & k_2 & k_3 \end{smallmatrix}\big] = -2\delta_{j,j_1,j_2,j_3} \delta_{k_1,k_2,k_3,0} c_{\alpha_{j}}$. For $(g,n) = (1,1)$, we have
\begin{equation*}
\begin{split} 
\omega_{1,1}(z_1) & = \sum_{\alpha \in \mathcal{R}} \Res_{z = 0} K_{\alpha}(z_1,\alpha + z) \frac{\dd z\,\dd \overline{z}}{(z - \overline{z})^2} \\
& = \sum_{\alpha \in \mathcal{R}}  \Res_{z = 0} \bigg(\frac{K_{\alpha,-1}(z_1)}{z} + K_{\alpha,0}(z_1) + K_{\alpha,1}(z_1) z + O(z^2)\bigg)\bigg(\frac{1}{4z^2} + \frac{a_1}{4z} + \frac{a_1^2 + 2a_2}{8} + O(z)\bigg)\dd z \dd z_1 \\
& = \frac{1}{4} K_{\alpha,-1}(z_1)+ \frac{a_1}{4} K_{\alpha,0}(z_1) + \frac{a_1^2 + 2a_2}{8} K_{\alpha,1}(z_1).
\end{split}
\end{equation*} 
in terms of the coefficients of expansion of the involution \eqref{barzoverz}.
Then, the Lemmata~\ref{lem:involution} and \ref{lem:kernel} guarantee that the entries of $G_{1,1}$ and $G_{0,3}$ have regular Puiseux series expansion as $p_r \rightarrow 0$.

\medskip

The topological recursion formula \eqref{Trformula} implies a recursion on $2g - 2 + n > 0$ for the $G_{g,n}$, which we can initialise with $G_{0,3}$ and $G_{1,1}$ --- this is closely related to the Airy structure form of the topological recursion \cite{KSTR,TRABCD}. To describe its structure, we first need to compute as $z \rightarrow 0$
\begin{equation}
\label{dun1}
\omega_{0,2}(z_i,\alpha + z)  = \frac{\dd z_i \dd z}{(z - z_i)^2}  \approx \frac{(k + 1)z^k}{(z_i - \alpha)^{k + 2}},
\end{equation}
and
\begin{equation}
\label{dun2}
\omega_{0,2}(z_i,\alpha + \overline{z}) \approx \sum_{k \geq 0} \frac{(k + 1)\overline{z}^k}{(z_i - \alpha)^{k + 2}} \approx \sum_{k \geq 0} z^k \left(\sum_{l = 0}^{k} \frac{-(l + 1)W_{\alpha,k,l}}{(z_i - \alpha)^{l + 2}}\right), 
\end{equation} 
where $W_{\alpha,k,l}$ are polynomial in the variables $a_1,a_2,\ldots$. Besides, for $2g - 2 + n > 0$ the relation between the correlators and the free energies \eqref{freener} yields:
\begin{equation}
\label{dun3}
\omega_{g,n}(z_1,\ldots,z_n) = \sum_{\substack{1 \leq j_1,\ldots,j_n \leq r \\ k_1,\ldots,k_n \geq 0}} C_{g,n}\big[\begin{smallmatrix} j_1 & \cdots & j_n \\ k_1 & \cdots & k_n \end{smallmatrix}\big] \bigotimes_{i = 1}^{n} \frac{-(k_i + 1)\dd z_i}{(z_i - \alpha_{j_i})^{k_i + 2}}.
\end{equation} 
Inserting \eqref{dun1}-\eqref{dun2}-\eqref{dun3} in the topological recursion formula \eqref{Trformula}, we get
\begin{equation*}
\begin{split} 
& \quad G_{g,1 + n}\big[\begin{smallmatrix} j_0 & \cdots & j_n \\ k_0 & \cdots & k_n \end{smallmatrix}\big] \\
& = \sum_{i = 1}^n \sum_{\substack{1 \leq j \leq r \\ k \geq 0}} B\big[\begin{smallmatrix} j_0 & j_i & j \\ k_0 & k_i & k \end{smallmatrix}\big] G_{g,n}\big[\begin{smallmatrix} j & j_1 & \cdots & \widehat{j_i} & \cdots & j_n  \\ k & k_1 & \cdots & \widehat{k_i} & \cdots & k_n \end{smallmatrix}\big] \\
& \quad + \sum_{\substack{1 \leq j,j' \leq r \\ k,k'\geq 0}}\frac{C\big[\begin{smallmatrix} j_0 & j & j' \\ k_0 & k & k' \end{smallmatrix}\big]}{2}\left(G_{g-1,n+2}\big[\begin{smallmatrix} j & j' & j_1 & \cdots & j_n \\ k & k' & k_1 & \cdots & k_n \end{smallmatrix}\big] +  \!\!\!\sum_{\substack{I \sqcup I' = \{1,\ldots,n\} \\ h + h' = g}} \!\! G_{h,1+|I|}\big[\begin{smallmatrix} j & (j_i)_{i \in I} \\ k & (k_i)_{i \in I}\end{smallmatrix}\big] \cdot G_{h',1+|I'|}\big[\begin{smallmatrix} j' & (j_i)_{i \in I'} \\ k' & (k_i)_{i \in I'} \end{smallmatrix}\big]\right),
\end{split}
\end{equation*} 
with the conventions $G_{0,1} = 0$ and $G_{0,2} = 0$. The recursion coefficients $B$ and $C$ are computed as the coefficients of $-(k_0 + 1)(z_0 - \alpha_{j_0})^{-(k_0 + 2)}$ in the respective expressions
$$
\delta_{j_0,j_i} \Res_{z = 0} K_{\alpha_{j_0}}(z_0,z + \alpha_{j_0})\left((k_i + 1)z^{k_i} \cdot \frac{-(k + 1)}{(\overline{z} + \alpha_{j_0} - \alpha_{j})^{k + 2}}  + \sum_{k' \geq k_i} z^{k'} W_{\alpha_{j_0},k',k_i} \cdot \frac{-(k + 1)}{(\overline{z} + \alpha_{j_0} - \alpha_{j})^{k + 2}}\right)
$$
and
$$
\Res_{z = 0} K_{\alpha_{j_0}}(z_0,\alpha_{j_0} + z) \left( \frac{(k + 1)}{(z + \alpha_{j_0} - \alpha_j)^{k + 2}} \cdot \frac{(k' + 1)}{(\overline{z} + \alpha_{j_0} - \alpha_{j'})^{k' + 2}} + (z \leftrightarrow \overline{z})\right).
$$
We had to distinguish between the $B$-terms and the $C$-terms due to the two different expansions for $\omega_{0,2}$ and the other $\omega_{g',n'}$. Only finitely many terms of the sum over $k'$ can contribute in $B$ because we are extracting a given power of $(z_0 - \alpha_{j_0})^{-1}$. With the help of Lemmata~\ref{lem:involution} and \ref{lem:kernel} and arguing as in the previous paragraphs, one can infer that $B$ and $C$ are polynomials in $\alpha^{-1}$ and $(\alpha - \gamma)^{-1}$ for $\alpha \neq \gamma$ in $\mathcal{R}$, and we deduce that they admit an regular Puiseux series expansion as $p_r \rightarrow 0$. By induction on $2g - 2 + n > 0$, we can then conclude that, under the Assumptions of Lemma~\ref{lem:kernel}, the $G_{g,n}$ also have regular Puiseux series expansion as $p_r \rightarrow 0$.

 \medskip

We are now in position to conclude the proof of Proposition~\ref{polynom3}. Recall that in the setting of this Proposition we have $p_{r - 1} = \cdots = p_{r - s + 1} = 0$. We have seen in \eqref{Hgn3rddd} that for any given partition $\mu$ of length $n$, the coefficient $H_{g,n}^{\text{3rd}}(\mu_1,\ldots,\mu_n)$ has a Puiseux expansion in the variable $p_r$, whose coefficients are polynomials in the variables $p_1,\ldots,p_{r - s}$. Let us focus on the coefficient of a fixed negative power of $p_r$ in a fixed $H_{g,n}^{\text{3rd}}(\mu_1,\ldots,\mu_n)$.  This coefficient is determined by the value it takes on finitely many tuples of complex numbers $(p_1,\ldots,p_{r - s})$; how many depends on $g,n$ and the power of $p_r$ one looks at. In particular, we choose this amount of finitely many tuples so that all of them satisfy the assumptions of the Lemma~\ref{lem:kernel}. We then conclude from the previous paragraph that the coefficient of the negative power of $p_r$ under consideration in $H_{g,n}^{\text{3rd}}(\mu)$ must vanish. This proves Proposition~\ref{polynom3}.

\vspace{1cm}
\section{Vanishing of single $\Omega$-integrals: theorems and experiments}
\label{sec:Omega:vanishing}
\vspace{0.5cm}

Vanishing of integrals of $\Omega$-classes usually arise for particular geometric reasons. In this section we discuss the single integral vanishing that were known and the new ones that we have experimentally found by computing many examples of the same form of the ones appearing in this article.

\medskip

A certain class of vanishing was obtained in the work of Johnson, Pandharipande, and Tseng \cite[Theorem 2]{JPT} with a geometric proof. The statement in the original paper is expressed in terms of integrals over the space of admissible covers, which becomes the integral of a $\Omega$-class after taking the pushforward to the moduli space of stable maps, as proved in \cite{lew-pop-sha-zvo17}. We give below the pushforwarded version adapted to our notation.

\begin{theorem}\label{thm:JPT} \cite{JPT} Let $r = s \in \mathbb{Z}_{> 0}$ and $g \in \mathbb{Z}_{\geq 0}$. Let $\mu$ be a non-empty partition, $n > 0$ and $b_1,\ldots,b_n \in \{0,\ldots,r - 1\}$ such that $|\mu| = \sum_{i = 1}^{n} b_j\,\,\text{mod}\,\,r$. Assume that \emph{at least one} of the following conditions hold:
\begin{itemize}
\item[(i)] The negativity condition $|\mu| < \sum_{i = 1}^n b_j$ and the boundedness condition $\max_{i \neq j}(b_i + b_j) \leq r$.
\item[(ii)] The strong negativity condition $(r - 1)|\mu| + \sum_{i = 1}^{n} b_i < rn$.
\end{itemize}
Then, we have the vanishing
\begin{equation}\label{eq:jptvanishing}
\int_{\MM_{g, \ell(\mu)+n}} \frac{\Omega^{(r,s)}_{g; -\overline{\mu}; b_1, \ldots, b_n}}{ \prod_{i=1}^{\ell(\mu)}(1 - \frac{\mu_i}{r}\psi_i)} = 0.
\end{equation}
\end{theorem}

The statement replacing Theorem~\ref{thm:JPT} in absence of (i) was obtained in \cite{BDKLM} and is the special case $d = 1$ of Theorem~\ref{th:A}: it says that a linear combination of $\Omega$-integrals vanish, and it contains a single term whenever the boundedness condition is satisfied. The strong negativity condition does not seem to play a role in our approach, so should be considered as a vanishing relation of a different nature than Theorem~\ref{th:A}.

\medskip

The theorem above holds as well for $s=0$, since it gives the same $\Omega$-class as $s = r$. The correct tuning of the parameter $s$, especially in relation with the $r$ parameter, has played an important role in the applications of $\Omega$-classes such as for instance Hurwitz theory \cite{kra-lew-pop-sha19,bor-kra-lew-pop-sha20,dun-kra-pop-sha19-2}, the double ramification cycle \cite{JPPZ17}, Masur-Veech volumes \cite{CMS+19}, and the Euler characteristic of $\mathcal{M}_{g,n}$ \cite{GLN20}. It is therefore natural to ask whether a generalisation of Theorem~\ref{thm:JPT} for general $s$ exists.

\medskip

For this purpose, we have run computations of integrals of the form \eqref{eq:jptvanishing} for  $\dim_{\mathbb{C}}(\MM_{g,n}) \leq 5$ using the Sage package \textsc{admcycles} \cite{admcycles}. First, the computations suggest that for $\mu = (1)$ one could drop conditions (i) and (ii) of Theorem \ref{thm:JPT}, in the sense that the integral in \eqref{eq:jptvanishing} vanishes nevertheless:
\begin{equation}\label{eq:initialguess}
\forall s \in \{0,r\}\qquad  \int_{\MM_{g,1+n}} \frac{\Omega^{(r,s)}_{g; r-1,b_1, \ldots ,b_n}}{\left(1 - \frac{1}{r}\psi_1\right)} = 0
\end{equation}
for any $b_1,\ldots,b_n \in \{0,\ldots,r-1\}$. Additionally, the vanishing \eqref{eq:initialguess} holds as well for $s=-1$. The case $s=-1$ is a particular case of the observed vanishing:
\begin{equation}\label{eq:s=-mu}
\forall s \in \{-r+1,\ldots,-1\}\qquad \int_{\MM_{g,1+n}} \frac{\Omega^{(r,s)}_{g; |s|,b_1, \ldots, b_n}}{\left(1 + \frac{s}{r}\psi_1\right)} = 0.
\end{equation} 

The latter can be proved as an specialisation of the statement (a) for $  - r < s < 0$ in the following vanishing result.
\begin{theorem} \cite{GLN20}
\label{thm:PADrel}
Let $r \in \mathbb{Z}_{> 0}$ and $s \in \mathbb{Z}$, let $g,n \in \mathbb{Z}_{\geq 0}$  such that $2g - 2 + n > 0$, and $b_1,\ldots,b_n \in \{0,\ldots,r-1\}$ such that $(2g - 2 + n)s = \sum_{i=1}^n b_i \,\,\text{mod}\,\,r$. Denote $\langle s \rangle \in \{0,\ldots,r-1\}$ the remainder of the Euclidean division of $s$ by $r$.  Then, the following properties hold
\begin{itemize} 
\item[(a)] If $s < 0$ we have:
\begin{equation}
\int_{\overline{\mathcal{M}}_{g,1+n}}
\!
		\frac{ \Omega^{(r,s)}_{g; \langle s \rangle, b_1, \ldots, b_n}}{
			\prod_{m = 0}^{-\lfloor s/r \rfloor - 1} \left( 1 -  \left( - \frac{s}{r} -  m \right)	 \psi_1 \right) 
		} = 0.
\end{equation}

\item[(b)] If $\; 0 \leq s \leq r$ we have:
\begin{equation}
\int_{\overline{\mathcal{M}}_{g,1+n}}
		\!\!\!\!\! \Omega^{(r,s)}_{g; s, b_1, \ldots, b_n}= 0.
\end{equation}

\item[(c)] If $s \geq r$ we have:
\begin{equation}
\int_{\overline{\mathcal{M}}_{g,1+n}}
		\!\!\!\!\! \Omega^{(r,s)}_{g; \langle s \rangle, b_1, \ldots, b_n}
			\prod_{m = 1}^{\lfloor s/r \rfloor} \left( 1 +  \left( \frac{s}{r} - m \right)	 \psi_1 \right) = 0.
\end{equation}
\end{itemize}
\end{theorem}

\medskip

Let us now consider the integrals 
\begin{equation}
\int_{\MM_{g,1+n}} \frac{\Omega^{(r,s)}_{g; -\overline{\mu}_1,b_1, \ldots, b_n}}{\left(1 - \frac{\mu_1}{r} \psi_1\right)} 
\end{equation} 
for fixed $g,n,r,s$ as functions of positive integers $\mu_1$. We have observed numerically that, for $s < 0$, these integrals not only vanish for $\mu_1 = -s$ as discussed above, but most of the vanishing happen around that value: for $\mu_1 = -s+1, -s+2$, etc. Sometimes, for a few values $\mu_1 < -s$, many of the tuples $(b_1, \ldots, b_n)$ produce vanishing, with the amount of tuples leading to vanishing fading out away from the value $\mu_1 = -s$. It can happen that for some $\mu_1 = -s + a$ with $a$ a small positive integer, \textit{all} tuples $b_1,\ldots,b_k$ give vanishing. If only some tuples produce vanishing and some do not, there seems to be an boundedness-type condition
$$
\max_{i \neq j} (b_i + b_j) \leq C_{g,n}^{(r,s)}(\mu_1)
$$
which is sufficient to produce vanishing --- we already know the specialisation $C_{g,n}^{(r,r)}(\mu_1) = r$ from Theorem~\ref{thm:JPT}. If such a bound exists it should depend in a non trivial way on $\mu_1$ and on the genus, and this phenomenon goes undetected for $s=r$. However, in a few cases, we encounter counterexamples. For instance, in genus one for $(r,s) = (6,-3)$ and $\mu_1 = 6$, we have computed in genus $1$
$$
\int_{\MM_{1,5}} \frac{\Omega^{(6,-3)}_{1; 0; 5,4,4,2}}{(1 + \psi_1)}  = 0, 
\qquad \qquad \qquad 
\int_{\MM_{1,5}} \frac{\Omega^{(6,-3)}_{1; 0; 5,4,3,3}}{(1 + \psi_1)}  \neq 0,
$$
despite the two tuples $\boldsymbol{b} = (5,4,4,2)$ and $\boldsymbol{b} = (5,4,3,3)$ having the same $\max \limits_{i \neq j}(b_i + b_j)$, and even the same $\sum_{i} b_i$. As a last observation, we have not encountered any vanishing for $s < -r$.

\newpage
\appendix

\section{Properties and symmetries of $\Omega$-classes}
\label{sec:Omega:properties}

This section contains a collection of previously known and basic properties of the $\Omega$-classes as defined in \eqref{eqn:Omega}.

\begin{theorem} \cite{GLN20} \label{thm:properties}
	Fix $g,n \in \mathbb{Z}_{\geq 0}$ such that $2g - 2 + n > 0$. Let $(r,s) \in \mathbb{Z}_{> 0} \times \mathbb{Z}$ and $a_1,\ldots,a_n \in \{0,\ldots,r-1\}$ such that $a_1+a_2+\cdots+a_n = (2g-2+n)s \,\,\text{mod}\,\,r$. The $\Omega$-classes satisfying the following properties
	\begin{itemize}
		\item[(i)]
		Shift of $s$:
		\begin{equation}
			\Omega^{(r,s)}_{g;a_1,\ldots,a_n}(u) = 
			\Omega^{(r,s)}_{g;a_1,\ldots,a_n}(u) \cdot \exp\Biggl( \sum_{m \ge 1} \frac{(-u)^m}{m} \left(\frac{s}{r}\right)^m \kappa_m \Biggr).
		\end{equation}
		\item[(ii)]
		Shift of $a_i$:
		\begin{equation}
			\Omega^{(r,s)}_{g;a_1,\ldots,a_{i - 1},a_i + r,a_{i + 1},\ldots,a_n}(u) 
			=
			\Omega^{(r,s)}_{g;a_1,\ldots,a_n}(u) \cdot\left( 1 + u\frac{a_i}{r}\psi_i\right).
		\end{equation}
		\item[(iii)]
		Zero and $r$-symmetry:
		\begin{equation}
		\begin{split}
			\Omega^{(r,0)}_{g;a_1,\ldots,a_n} & = \Omega^{(r,r)}_{g;a_1,\ldots,a_n},
	\\
			\Omega^{(r,s)}_{g;a_1,\ldots,a_{i - 1},0,a_{i + 1},\ldots,a_n} & = \Omega^{(r,s)}_{g;a_1,\ldots,a_{i - 1},r,a_{i + 1},a_n}.
		\end{split}
		\end{equation}
		\item[(iv)]
		Pullback property. Denoting $\pi : \overline{\mathcal{M}}_{g,n + 1} \rightarrow \overline{\mathcal{M}}_{g,n}$ the forgetful morphism, we have
		\begin{equation}
			\Omega^{(r,s)}_{g;a_1,\ldots,a_n} = \pi^{\ast} \Omega^{(r,s)}_{g;a_1,\ldots,a_n,s}.
			\end{equation}
		\item[(v)]
		\textup{(String equation).} For formal variables $u_1, \ldots, u_{n+1}$ we have:
		\begin{equation}
			\int_{\overline{\mathcal{M}}_{g, n+1}} \frac{\Omega^{(r,s)}_{g;a_1,\ldots,a_n,s}}{\prod_{i=1}^{n+1} (1 - u_i \psi_i)} \Bigg{|}_{u_{n+1} = 0}
				=
				(u_1 + \cdots + u_n) \int_{\overline{\mathcal{M}}_{g, n}} \frac{\Omega^{(r,s)}_{g;a_1,\ldots,a_n}}{\prod_{i=1}^n (1 - u_i \psi_i)}.
		\end{equation}

		\item[(vi)] 
		\textup{(Dilaton equation).} For formal variables $u_1, \ldots, u_{n+1}$ we have:
		\begin{equation}\label{eqn:dilaton:Omega}
			\frac{\partial}{\partial u_{n+1}}
			\int_{\overline{\mathcal{M}}_{g, n+1}} \frac{\Omega^{(r,s)}_{g;a_1,\ldots,a_n,s}}{\prod_{i=1}^{n+1} (1 - u_i \psi_i)} \Bigg{|}_{u_{n+1} = 0}
			=
			(2g - 2 + n) \int_{\overline{\mathcal{M}}_{g, n}} \frac{\Omega^{(r,s)}_{g;a_1, \ldots, a_{n}}}{\prod_{i=1}^n (1 - u_i \psi_i)}.
		\end{equation}
	\end{itemize}
\end{theorem}
	
Iterating (i) and (ii), one finds for any $N \in \mathbb{Z}_{> 0}$
	\begin{enumerate}[(I)]
		\item\label{prop:mult:shifts:s}
		Multiple shifts of $s$:
$$
\Omega^{(r,s + Nr)}_{g;a_1,\ldots,a_n}(u) = \Omega^{(r,s)}_{g;a_1,\ldots,a_n}(u) \cdot \exp\Biggl( \sum_{m \ge 1} \frac{(-u)^m}{m} p_m\left(\frac{s}{r}, \ldots, \frac{s}{r} + N - 1 \right) \kappa_m \Biggr),
$$
		where $p_m$ is the sum of $m$-th powers.
		\item\label{prop:mult:shifts:ai}
		Multiple shifts of $a_i$:
		\begin{equation}
			\Omega^{(r,s)}_{g;a_1,\ldots,a_{i - 1},a_i + Nr,a_{i + 1},\ldots,a_n}(u)
			=
			\Omega^{(r,s)}_{g;a_1,\ldots,a_n}(u) \cdot \prod_{m = 0}^{N-1} \left( 1 + u\left(\frac{a_i}{r} + m \right) \psi_i\right).
		\end{equation}
	\end{enumerate}

\medskip

Another interesting property, which only holds for $r=1$, is a relation between two different parametrisations of $\Omega$-classes, that may be called \emph{Segre} and \emph{Chern}:
\begin{equation}
\Omega^{(1,1-s)}_{g;0,\ldots,0}(u) = \big(\Omega^{(1,s)}_{g;0,\ldots,0}(u)\big)^{-1}.
\end{equation}
It has been proved and used in \cite{CMS+19}. The relation that one might expect from Serre duality applied to an $r$-th root of $\omega_{\log}^{\otimes s}\big(-\sum_{i = 1}^{n} a_i p_i\big)$, \textit{i.e.} $\Omega^{(r,r - s)}_{g;r-a_1,\ldots,r-a_n}(u) = \big(\Omega^{(r,s)}_{g;a_1,\ldots,a_n}(u)\big)^{-1}$, is in fact false. As a counterexample, in topology $(g,n) = (1,2)$ we have
$$
\Omega^{(2,1)}_{g=1;0,2}(u) \cdot \Omega^{(2,1)}_{g=1;2,0}(-u) =  1- \frac{3}{4}u^2\kappa_2.
$$
Nevertheless, in our numerical experiments we have observed the vanishing
$$
\forall k \in \mathbb{Z}_{\geq 0} \qquad \big[\Omega^{(r,r-s)}_{g;r - a_1,\ldots,r - a_n}(-u) \cdot \Omega^{(r,s)}_{g;a_1,\ldots,a_n}(u)\big]_{\text{deg}\,2(2k + 1)} = 0.
$$

\newcommand{\etalchar}[1]{$^{#1}$}

\end{document}